\documentclass[11pt,reqno,tbtags]{amsart}
\usepackage{graphicx}
\usepackage{amsmath,amsthm,amssymb}
\usepackage{amstext}
\usepackage{alltt} 
\usepackage{array,longtable}
\usepackage[T1]{fontenc}
\usepackage[english]{babel}
\usepackage[pdftex]{hyperref}

\usepackage{multirow}

\usepackage{lineno}

\newcommand {\eio}{ edge insertion operation }

\theoremstyle{plain}
\newtheorem{theorem}{Theorem}[section]
\newtheorem{lemma}[theorem]{Lemma}
\newtheorem{proposition}[theorem]{Proposition}
\newtheorem{corollary}[theorem]{Corollary}
\newtheorem{conjecture}[theorem]{Conjecture}
\newtheorem{sconjecture}[theorem]{Refuted Conjecture}
\newtheorem{observation}[theorem]{Observation}

\theoremstyle{definition}

\newtheorem{problem}{Problem}

\theoremstyle{remark}

\begin{document}

\title{Generation and Properties of Snarks}

\author{Gunnar Brinkmann}
\address{Applied Mathematics \& Computer Science,  Ghent University,Krijgslaan 281-S9, \\9000 Ghent, Belgium}
\email{Gunnar.Brinkmann@UGent.be}
\author{Jan Goedgebeur}
\address{Applied Mathematics \& Computer Science,  Ghent University,Krijgslaan 281-S9, \\9000 Ghent, Belgium}
\email{Jan.Goedgebeur@UGent.be}
\author{Jonas H{\"a}gglund}
\address{Department of Mathematics and Mathematical Statistics, Ume\aa{} universitet, S-901 87 Ume\aa, Sweden}
\email{Jonas.Hagglund@math.umu.se}
\author{Klas Markstr{\"o}m}
\address{Department of Mathematics and Mathematical Statistics, Ume\aa{} universitet, S-901 87 Ume\aa, Sweden}
\email{Klas.Markstrom@math.umu.se}

\begin{abstract} 
	For many of the unsolved problems concerning cycles and matchings in graphs  it is known that it is sufficient to prove them 
	for \emph{snarks}, the class of nontrivial 3-regular graphs which cannot be 3-edge coloured.
	
	In the first part of this paper we present a new algorithm for generating all non-isomorphic snarks of a given 
	order. Our implementation of the new algorithm is 14 times faster than previous programs for generating 
	snarks, and 29 times faster for generating weak snarks. Using this program we have generated all 
	non-isomorphic snarks on $n\leq 36$ vertices. Previously lists up to $n=28$ vertices have been published.
	
	In the second part of the paper we analyze the sets of generated snarks with respect to a number of properties 
	and conjectures.  We find that some of the strongest versions of the cycle double cover conjecture hold for all 
	snarks of these orders, as does Jaeger's Petersen colouring conjecture, which in turn implies that Fulkerson's 
	conjecture has no small counterexamples. In contrast to these positive results we also find counterexamples to eight previously 
	published conjectures concerning cycle coverings and the general cycle structure of cubic graphs.
\end{abstract} 

\maketitle

\tableofcontents

%--------------------------------------------------------------------------------------------------------------------------------------------------------------------------------------------
\section{Introduction}
A number of problems in graph theory can be solved in the general case
if they can be solved for cubic, or 3-regular, graphs. Examples of
such problems are the four colour problem, now a theorem, many of the
problems concerning cycle double covers and surface embeddings of
graphs, coverings by matchings, and the general structure of the cycle
space of a graph. For most of these problems one can additionally
constrain this class of graphs to the subclass of cubic graphs which
cannot be 3-edge coloured.  By the classical theorem of Vizing, a
cubic graph has chromatic index 3 or 4. Isaacs~\cite{isaacs_75} called
cubic graphs with chromatic index 3 \textit{colourable} graphs and
those with chromatic index 4 \textit{uncolourable} graphs.
Uncolourable cubic graphs with cycle separating 2- or 3-cuts, or
4-cycles can be constructed from smaller uncolourable graphs by
certain standard operations, which also behave well with respect to
most of the open problems at hand. Thus minimal counterexamples to
many problems must reside, if they exist at all, among the remaining
uncolourable cubic graphs, which are called \emph{snarks} based on an article by
Gardner~\cite{gardner_76} who used the term with weaker connectivity
requirements.
Later on, stronger criteria for
non-triviality have also been proposed \cite{snark_reduct,BrSt95_2}
but here we will focus on snarks and their properties.

So far no simple way to study the full class of snarks has been found,
e.g. there is at present no known uniform random model for snarks,
hence leaving us without a theoretical method for studying the typical
behaviour of snarks. The only available alternatives have been to
study the smallest snarks and certain families of snarks given by
specialized constructions.  Lists of snarks have already been given in
\cite{BrSt95_2,1st_snarklist,snarklist2,snarklist3}, but so far no
specialized computer program for generating snarks existed. The
fastest program was the one used in \cite{BrSt95_2} which was based on
the program described in \cite{Br96_1} and -- just like the approach
in \cite{snarklist3} -- is simply a generator for all cubic graphs
with a lower bound on the girth, combined with a filter for
colourability at the end. However, since the proportion of snarks
among the cubic graphs rapidly decreases as the number of vertices
increase, and is asymptotically zero, this approach is not feasible
for sizes even a few steps beyond those previously published.  In this
paper we give a new algorithm which augments an efficient algorithm
for generating cubic graphs, given in \cite{tricycle}, with a
look-ahead method which makes it possible to avoid constructing many
of the colourable cubic graphs, thus reducing the number of graphs
passed to the final filter step.  Our program based on this new
algorithm is 14 times faster than previous programs for generating
snarks and the speedup is increasing with the size of the generated
snarks.  Using a parallel computer for approximately 73 core-years, we
have generated all non-isomorphic snarks on $n\leq 36$ vertices.

In addition to generating all snarks up to 36 vertices we have
performed an extensive analysis of the generated snarks.  In Section 3
we present the results of this analysis. We have classified the snarks
according to a number of previously studied properties and tested a
number of conjectures, most of them related to the cycle double cover
conjecture. Many of the conjectures were found to hold for snarks of
these orders but a number of counterexamples were also found.  Let us
mention a few of these examples. We have found examples of permutation
graph snarks on 34 vertices which are cyclically 5-edge connected,
thus disproving the conjecture from \cite{MR1426132} stating that the
Petersen graph is the only such snark. We have found examples giving a
negative answer to a question from \cite{MR1239236} regarding whether
every 2-regular subgraph of a cyclically 5-edge connected snark is a
subset of a cycle double cover. We have found counterexamples to a
family of four conjectures equivalent to a conjecture from
\cite{MR2355128} regarding compatible cycle decompositions of
4-regular graphs. In the positive direction we have also used our
verification of the strong cycle double cover conjecture for our set
of snarks to prove that any $n$ vertex snark with a cycle of length at
least $n-10$ has a cycle double cover.

To conclude, the results of the analysis of the snarks of order at
most 36 show that some of the intuition gained from the study of the
smallest snarks has been misleading. Some of the unexpected behaviour
of snarks can only be found for sufficiently large snarks. As we will
see, 34 vertices seems to be a change point for many properties.  A
look at how fast the number of snarks grows also shows a change at 34
vertices, indicating that we might have just reached the size range
where the snarks display a behaviour more typical for large snarks.
We believe that the results found also give a good demonstration of
the importance of large scale computer based generation and analysis
as tools for aiding our understanding of combinatorial problems.

%--------------------------------------------------------------------------------------------------------------------------------------------------------------------------------------------
\section{Defintions}
A \emph{cycle} is a 2-regular connected graph. The \emph{circumference} of a graph $G$ is the number of vertices in a longest cycle in $G$. The \emph{girth} of a graph is the number of vertices in a  shortest cycle in $G$ and is denoted  $g(G)$.

A graph $G$ is \emph{cyclically $k$-edge connected} if the deletion of fewer than $k$ edges from $G$ does not create two components both of which contain at least one cycle.  The largest integer $k$ such that $G$ is cyclically $k$-edge connected is called the cyclic edge-connectivity of $G$ and is denoted $\lambda_c(G)$.

A $2$-factor of a graph $G$  is a spanning 2-regular subgraph of $G$.  The \emph{oddness} of a bridgeless cubic graph is the minimum number of odd order  components in any 2-factor of the graph.

A \emph{weak snark} is an uncolourable cyclically 4-edge connected cubic graph with girth at least 4. A \emph{snark} is an uncolourable cyclically 4-edge connected cubic graph with girth at least 5.

%--------------------------------------------------------------------------------------------------------------------------------------------------------------------------------------------	
 \section{Generation of snarks}
\subsection{The Generation Algorithm}
As mentioned in the introduction, the older programs for generating
snarks were built by adding a filter for graphs with the desired
properties to a program which generates all cubic graphs of a given
order.  The efficiency of an approach that generates a larger class of
graphs and filters the output for graphs in a smaller class depends on
one hand on the cost for the filter and on the other on the ratio
between the number of graphs in the large and small class. On both
criteria the generation of all cubic graphs with girth $4$ and
filtering for weak snarks scores badly: testing for 3-colourability is
NP-complete and already for 28 vertices only 0.00015\% of the cubic
graphs with girth at least 4 are weak snarks (and only 0.00044\% of
the cubic graphs with girth at least 5 are snarks). In fact the ratio
is even decreasing with the number of vertices. Nevertheless no better
way to generate weak snarks was known until now. In this paper we
present a method that -- although not generating only weak snarks --
at least allows a limited look-ahead and increases the ratio of weak
snarks among the graphs with the largest number of vertices by a
factor of 85 to about 0.0135\% for 28 vertices (and by a factor of 20
to about 0.0076\% for snarks). But also here the ratio of (weak)
snarks is decreasing with the number of vertices

The new generation algorithm is based on the algorithm described in
\cite{tricycle}. That algorithm was developed for the generation of
all cubic graphs without girth restriction, but it can also generate
cubic graphs with girth at least 4 or 5 efficiently. In that algorithm
pairwise non-isomorphic connected cubic graphs are generated from
$K_4$ by first building {\em prime} graphs with the operations
(a),(b),(c) in Figure ~\ref{fig:cubeop} and afterwards applying
operation (d). For details on how the algorithm makes sure that no
isomorphic copies are generated, see \cite{tricycle}.

\begin{figure}[h!t]
	\centering
	\includegraphics[trim=0mm 0mm 0mm 0mm, clip, width=1.0\textwidth]{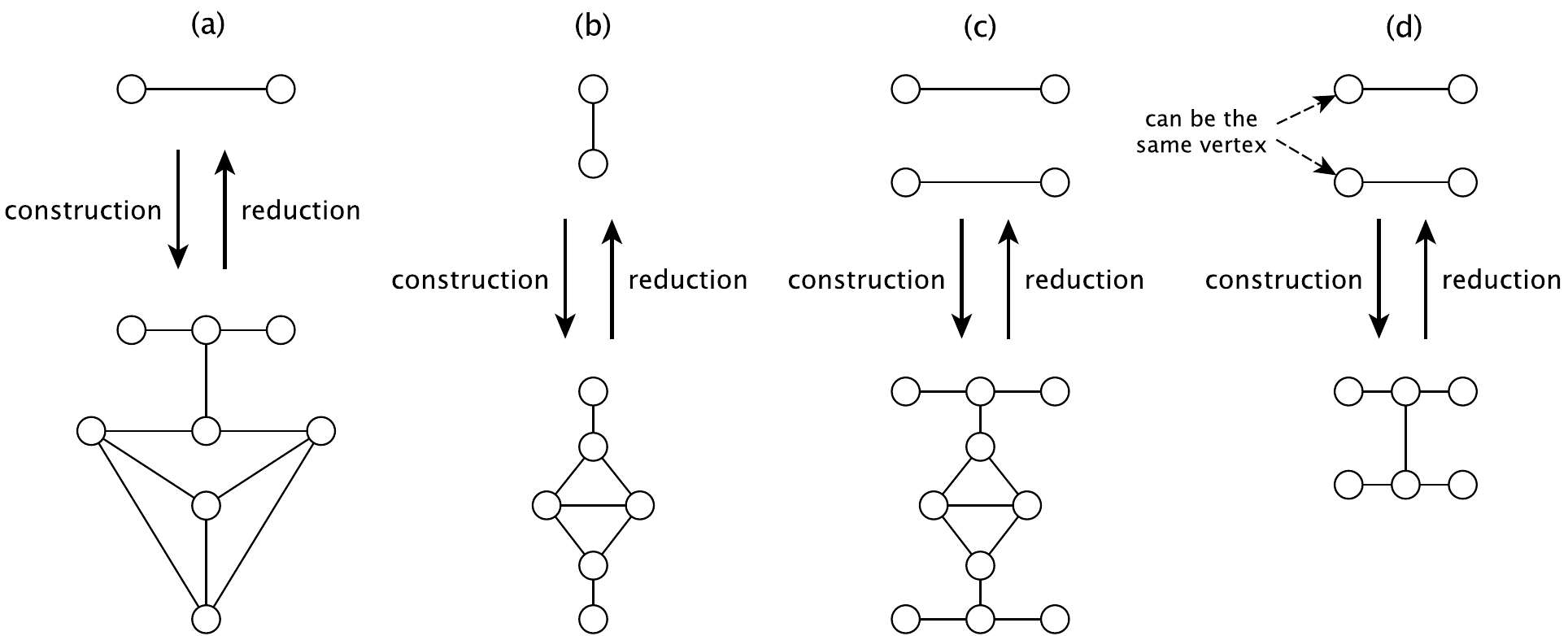}
	\caption{The construction operations for cubic graphs.}
	\label{fig:cubeop}
\end{figure}

Prime graphs are never 3-connected, so the last operation will always
be of type (d) given in Figure~\ref{fig:cubeop}. Triangles are
constructed by bundling these operations (for details see
\cite{tricycle}), but since weak snarks have girth at least 4, the
last operation will always be a single edge-insertion operation
applied to two disjoint edges.

The new algorithm implements a look-ahead that can often 
decide that an \eio would lead to a colourable graph and can therefore be avoided. 
As the number of cubic graphs grows fast with the 
number of vertices, just avoiding the insertion of the last edge gives already a considerable speedup.

We use the following well known result.
\begin{lemma} \label{theorem_2-factor}
	A cubic graph is colourable if and only if it has oddness 0. That is, it has a 2-factor with all cycles of even length, also called an even 2-factor.
\end{lemma} 

This is applied in the following way:

\begin{lemma}\label{lem:2factoredge}
	Given an even 2-factor $F$ in a cubic graph $G$. All graphs $G'$ obtained by applying the \eio to two edges $e,e'$ which are part of the same cycle 
	in $F$ will be colourable.
\end{lemma}

Replacing $e$ resp. $e'$ in $F$ with the two edges that result from the subdivision gives an even 2-factor in $G'$.

Lemma~\ref{lem:2factoredge} is applied by using that for any two colours $a\not= b$ the sets of edges coloured $a$ or $b$ form a 2-factor
with all cycles of even length: 
\begin{theorem}\label{theorem_add_edge_cycle}
	Given a cubic graph $G$ and a 3-colouring of $G$. If two edges
	$e,e'$ belong to the same cycle of a 2-factor induced by two different
	colours, the graph $G'$ obtained by applying the \eio to $e,e'$
	will be colourable.
\end{theorem}

Assume that snarks on $n$ vertices are to be constructed and a graph
$G$ on $n-2$ vertices has to be extended. Then the algorithm first
tries to construct a 3-colouring of $G$. If $G$ does not have a
3-colouring, Theorem~\ref{theorem_add_edge_cycle} can not be
applied. But the ratio of uncolourable bridgeless cubic graphs 
is very small. If a 3-colouring is found, then the complement
of each colour class is an even 2-factor. So each 3-colouring of $G$
gives three even 2-factors. When computing the edge pairs for
extensions, edge pairs with both edges in the same cycle of one of
these 2-factors are not considered.

Computing different 3-colourings -- that is 3-colourings producing
different 2-factors -- allows to detect more edge pairs that do not
lead to snarks. We compute different 3-colourings in two ways:
interchanging the colours in a non-hamiltonian cycle in one of the
2-factors gives a colouring where two of the three induced 2-factors
are different. So this is an efficient way to get a different
3-colouring. We use one such modified colouring. Another way is to
compute a different colouring from scratch by forcing some initial
colours in order to guarantee a different colouring.

For 28 vertices the first colouring allows in average to discard $85\%$ of the edge pairs. The colour changes
in the non-hamiltonian cycles discards another $11\%$. A second colouring discards $3.4\%$.
The cost for computing a third colouring turned out to be higher than the gain.

The following Theorem~\ref{theorem_square} gives another criterion to
avoid the generation of colourable graphs. It is folklore in
the community studying non-triviality of snarks, so we will not give a
proof here.

\begin{theorem}\label{theorem_square}
	Given a colourable graph $G$, all graphs $G'$ obtained from $G$ by the \eio so that
	the edge with the two new end-vertices is part of a 4-cycle in $G'$ are colourable.
\end{theorem}

So if we want to generate all snarks with $n$ vertices, we do not have
to apply the edge operation to colourable graphs with $n-2$ vertices
if the inserted edge will be part of a square. 

At first sight Theorem~\ref{theorem_square} looks only interesting for
operations that do not produce triangles, because in the last step we
never produce triangles. But in fact for triangles it allows an
earlier look-ahead. Though for details of isomorphism rejection we
refer the reader to \cite{tricycle}, one fact is important in this
context: the last edge inserted is always an edge in the smallest
cycle.

So assume that we have a graph $G$ with $n-2$ vertices and at least
one triangle and that we want to construct weak snarks with $n$ vertices.
From this graph we can only get graphs with girth at most $4$ -- so
the last edge inserted will be in a 4-gon. If $G$ was colourable,
all descendants will also be colourable and we don't have to construct
colourable graphs with triangles of size $n-2$. So for triangles
Theorem~\ref{theorem_square} gives a bounding criterion that can
already be applied on level $n-4$. In fact the generation algorithm
{\em bundles} triangle operations (see \cite{tricycle}), but this is
only a difference in detail.

%-------------------------------------------------------------------------------	
\subsection{Testing and running times}
We compared all snarks up to 32 vertices with the snarks which were generated by the program {\em minibaum} \cite{BrSt95_2}. The {\em minibaum} computation was independently run twice, on different machines in Belgium and Sweden. The results were in complete agreement.

For all snarks generated, we used an independent program to check the chromatic index and whether they are all pairwise non-isomorphic. As before this was done independently on separate machines  in Belgium and Sweden. The chromatic index was also checked a third time by computing the oddness of all generated graphs, see section \ref{sec:2fac}.

In Table \ref{tab:time} we present the running times for generating all snarks and weak snarks of that order, for $n\leq 34$. These values were obtained on a machine with an Intel Xeon L5520 processor running at 2.27 GHz. 
\begin{table}[h!]
		\begin{tabular}{r || rrrr}
		Order & weak snarks & snarks  & speedup for snarks\\
		\hline
		20 	& 0.1                   & 0.1                	   & 8.00\\
		22  	& 1.5                   & 1                	   &10.90\\
		24  	& 23.6                 & 15.3	            &11.86\\
		26  	& 431                  & 	272            &13.04\\
		28  	& 8873                & 	5304	          &14.18\\
		30  	&  209897          & 128875 	 &\\
		32   	&  4976553        & 2875911 	&\\	
		34 	&  119586562   & 	66519829 &\\
		\end{tabular}
		\caption{Running time in seconds. The speedup value is the quotient of the running time for minibaum and that of snarkhunter, for snarks.}
		\label{tab:time}
\end{table}

The total running time for  generating all snarks on $n=36$ vertices was approximately 73 core years, primarily running on a cluster with Intel Harpertown 2.66 GHz CPUs. An estimate based on running a partial generation for $n=38$ is that generating all snarks on this order would require 1100 core years of computing time, on the  clusters used for $n=36$.

The generator described here is called \textit{snarkhunter} and can be downloaded from \verb+ http://caagt.ugent.be/cubic/+.
 
%--------------------------------------------------------------------------------------------------------------------------------------------------------------------------------------------
\section{The number of snarks and their properties}
We have generated all snarks of order up to 36 and all weak snarks of
order 34 and less. The number of snarks with given girth and cyclic
connectivity can be found in Table \ref{sfig1}.  For each of the
properties in the coming sections testing programs were written
independently by the groups in Belgium and Sweden. The programs were
run on separate computers in the two countries in order to have
verification via different programs as well as different hardware.
All tests were in complete agreement for $n \leq 30$. For some of the
more time consuming tests only one group ran the full test for larger
$n$.
\begin{table}
		\begin{tabular}{r || rrrrrr}
		Order & weak & $\lambda_c \geq 4$ &  $\lambda_c \geq 5$ & $\lambda_c \geq 6$ & $g\geq 6$ & $g\geq 7$\\
		\hline
		10 	& 1     & 1 				& 1 		& 0	\\
		12,14,16 & 0 	  & -		         & -		& -	\\
		18 	& 2     & 2 				& 0		& - 	\\
		20 	& 6     & 6 				& 1 		& 0 	\\
		22  	& 31   & 20				& 2		& 0	\\
		24  	& 155     & 38				& 2		& 0	\\
		26  	& 1 297         & 280			& 10		& 0	\\
		28  	& 12 517       & 2 900			& 75		& 1 & 1 & 0	\\
		30  	& 139 854     & 28 399		& 509 	& 0	& 1 & 0 \\
		32   	& 1 764 950   & 293 059		& 2953 	& 0 & 0 & 0 \\	
		34 	& 25 286 953 & 3 833 587 		& 19 935	& 0& 0& 0\\
		36    & ? & 60 167 732 & 180 612	& 1 & 1& 0\\
		\end{tabular}
		\caption{The number of snarks.}
		\label{sfig1}
\end{table}

From \verb+http://hog.grinvin.org/Snarks+
all graphs listed in Table \ref{sfig1} can be downloaded.

%It is known that there are snarks of arbitrarily high girth \cite{Ko:96}.  At the moment the size of the smallest snark with girth 7 is unknown and it seems unlikely that the construction from \cite{Ko:96} is optimal. 
%\begin{problem}
%	Which is the smallest snark with girth at least 7?
%\end{problem}

%-------------------------------------------------------------------------------	
\subsection{Hypohamiltonian snarks}
A graph $G$ is \emph{hypohamiltonian} if $G$ is not hamiltonian but
$G-v$ is hamiltonian for every vertex $v$ in $G$.  Hypohamiltonian
snarks have been studied by a number of authors in connection with the
cycle double cover conjecture, see Section \ref{sec:cdc}, and
Sabidussi's Compatibility Conjecture \cite{FH09}. In \cite{CMRS} the
number of hypohamiltonian snarks on $n\leq 28$ vertices were given.
	
The number of hypohamiltonian snarks on $n\leq 36$ vertices can be
found in Table \ref{numhypo}, as well as the number of those which are
also cyclically 5-edge connected. We note the sudden increase in both
numbers at 34 vertices.
\begin{table}
		\begin{tabular}{r ||   rr | rr | c}
	\multirow{2}{*}{Order}  & \multicolumn{2}{c}{hypohamiltonian} \vline & \multicolumn{2}{c}{permutation} \vline & \multirow{2}{*}{circumference $n-2$}\\
		 & total & $\lambda_c \geq 5$& total & $\lambda_c \geq 5$& \\
		\hline
		10 	& 1 		&  1	& 1 & 1 & 0\\
		18 	& 2 		&  0 	& 2 & 0 & 0\\
		20 	& 1 		&  1 	& 0 & 0 & 0\\
		22  	& 2		&  2	& 0 & 0 & 0\\
		24  	& 0		&  0	& 0 & 0 & 0 \\
		26  	& 95		&  8	& 64 & 0 & 0\\
		28  	& 31		&  1	& 0 & 0 & 3\\
		30  	& 104	&  11& 0 & 0 & 8\\
		32   	& 13		& 13 & 0 & 0 & 32\\
		34 	& 31198	& 1497 & 10771 & 12& 104\\
		36	& 10838	& 464 & 0 & 0& 1143\\
		\end{tabular}
		\caption{The numbers of hypohamiltonian snarks, permutation snarks, and snarks with circumference $n-2$.}
		\label{numhypo}	\label{numperm}   \label{tabcirc}
\end{table}

%-------------------------------------------------------------------------------		
\subsection{Permutation snarks}
A cubic graph is a \emph{permutation graph} if it has a 2-factor that consists of two induced cycles. These graphs have also been called cycle permutation graphs \cite{PS81} and generalized prisms \cite{klee}.  A cubic graph is called a \emph{permutation snark} if it is both a snark and a permutation graph. 
	
It is obvious that all permutation snarks must be of order $2\mod 4$ since otherwise the graph would have a 2-factor that consists of two even components and therefore be colourable. The number of small permutation snarks can be found in Table \ref{numperm}. Looking at  Table \ref{numperm} an obvious question arises:
\begin{problem}
	Do all permutation snarks have $8k+2$ vertices, where $k=1, 2, 3,...$?
\end{problem}
It is well known that the Petersen graph is a permutation snark and in \cite{MR1426132} it was conjectured that this is the only permutation snark which is also cyclically 5-edge connected.
\begin{sconjecture}[Zhang \cite{MR1426132}]\label{conj:zhang}
	Let $G$ be a cubic cyclically 5-edge-connected permutation graph. If $G$ is a snark, then $G$ must be the Petersen graph. 
\end{sconjecture}
However, there are 12 cyclically 5-edge-connected permutation snarks on 34 vertices which provide counterexamples to  Conjecture \ref{conj:zhang}. One of these graphs can be found in Figure \ref{cyc5perm} and the full set is given in Appendix \ref{app:ex1}. They can also be obtained from \textit{House of Graphs}~\cite{HOG} by searching for the keywords ``counterexample * zhang''.
\begin{observation}\label{obs:perm12}
	Conjecture \ref{conj:zhang} is false. The smallest counterexamples have 34 vertices, and there are exactly 12 counterexamples of that order.
\end{observation}
These 12 graphs also have some other interesting properties as we shall see in Section \ref{s5}.		
\begin{figure}[h!]	
	\includegraphics[scale=0.3]{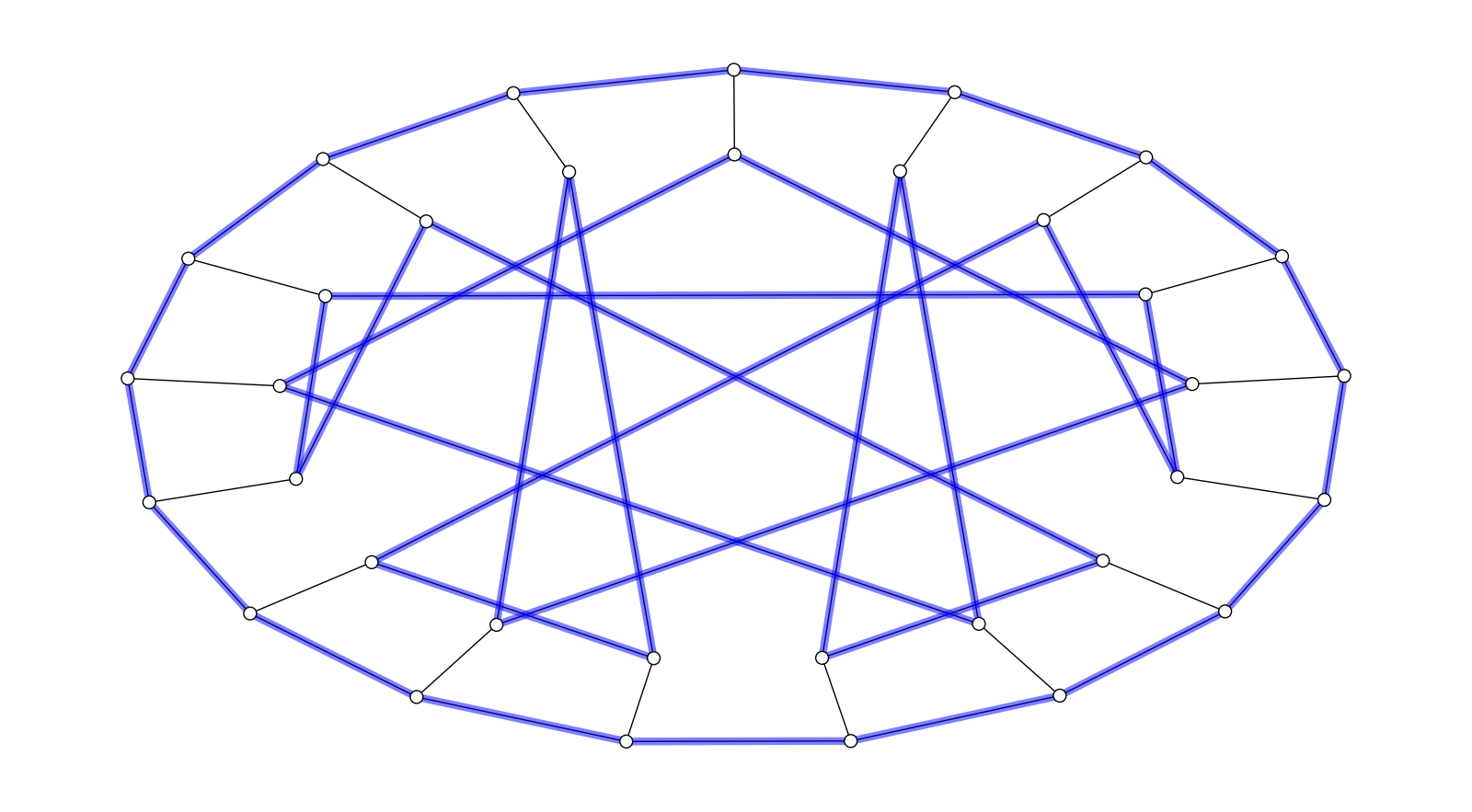}
		\caption{One of the twelve cyclically 5-edge-connected permutation snarks on 34 vertices. The bold cycles correspond to the 2-factor.}
	\label{cyc5perm}
\end{figure}

Goddyn has made a conjecture on the structure of permutation graphs. A cycle $C$ in a permutation graph $G$, with a given defining 2-factor, is said to be \emph{removable} if the graph obtained from $G$ by deleting all edges which belong to both $C$ and the induced 2-factor is 2-connected. Goddyn conjectured the following (see \cite{GHM}):
\begin{conjecture}[Goddyn \cite{GHM}]\label{conj:godd}
	The only permutation graph with no removable cycles for any defining 2-factor is the Petersen graph.
\end{conjecture}
It is known that a counterexample to the conjecture must be
uncolourable, \cite{AGZ}, and cannot have 4-cycles, since that would
make them hamiltonian and hence colourable.  We have tested the small
permutation snarks for the presence of removable cycles.
\begin{observation}\label{obs:rem}
		The Petersen graph is the only permutation snark on $n\leq 36$ vertices with no removable cycles.  The minimum number of removable cycles in a permutation snark  on $n$ vertices is 2 for $n=18$, 2  for $n=26$ and 3 for $n=34$. The permutation graphs on $n=34$ vertices with the largest number of removable cycles are the 12 graphs in  Observation \ref{obs:perm12}.
\end{observation}
	
%-------------------------------------------------------------------------------	
\subsection{Circumference and long  cycles}\label{sec:circum}

Obviously a snark must have circumference at most $n-1$ and in fact every snark of order $n \leq 26$ has circumference exactly $(n-1)$. The number of snarks with circumference $(n-2)$ can be found in Table \ref{tabcirc}. No snarks of order $n\leq 36$ have circumference less than $n-2$. 
		  
A \emph{dominating} cycle is a cycle $C$ such that every edge in $G$ has an endpoint on $C$.  In \cite{MR988642} Fleischner posed the following conjecture.
\begin{conjecture}[Fleischner \cite{MR988642}]\label{conj:doms}
	Every weak snark has a dominating cycle. 
 \end{conjecture}
We have verified this conjecture for all snarks on 36 vertices and all weak snarks on 34 vertices and less. 
\begin{observation}
	Conjecture \ref{conj:doms} has no counterexample on $n\leq 34$ vertices.
\end{observation}
It is not  obvious that the uncolourability matters for dominating cycles and there is a similar conjecture for general cubic well connected graphs:
\begin{conjecture}[Fleischner, Jackson, Ash \cite{MR988642}]\label{d4cs}
  	Every cubic cyclically 4-edge-connected graph has a dominating cycle. 
\end{conjecture}
In \cite{FK02} it was shown that Conjecture \ref{d4cs} is equivalent to the following, at a first look stronger, conjecture: 
\begin{conjecture}[Fleischner, Kochol \cite{FK02}]\label{d4cs2}
  	Given two edges $e_1$ and $e_2$  in a cubic cyclically 4-edge-connected graph $G$, there exists a  dominating cycle $C$ in $G$, which contains both $e_1$ and $e_2$. 
\end{conjecture} 
Even though the conjectures are equivalent, the smallest counterexamples to the different forms may have different sizes, so we have made independent tests for both forms. 
\begin{observation}
	There are no counterexamples to Conjectures \ref{d4cs} and \ref{d4cs2} among the snarks on $n\leq 36$ vertices, the weak snarks on $n \leq 34$ vertices and the general cubic cyclically 4-edge-connected graphs on $n \leq 26$ vertices.
	
\end{observation}

%-------------------------------------------------------------------------------	
\subsection{The structure of 2-factors}\label{sec:2fac}
The oddness of a cubic graph provides a measure for how far the
graph is from being colourable, as a graph has oddness 0 if and only
if it is colourable. The oddness has been used in the study of cycle
double covers in \cite{MR1328296,MR1815603,MR2117938,markstrom10},
where it was shown that cubic graphs with oddness at most 4 have
various types of cycle double covers.

Among the small snarks the oddness is very low.
\begin{observation}
	All snarks on $n\leq 36$ vertices have oddness 2. 
\end{observation}
We do not know the smallest snark of oddness greater than 2 and pose the following problem:
\begin{problem}
	Which is the smallest snark of oddness strictly greater than 2?
\end{problem}
By construction we know that there is a snark with oddness 4 on 46 vertices, \cite{Hagglund11:1}, but it seems unlikely that this is the smallest snark with that property. 

In \cite{ALS} Abreu et al. made the following conjecture:
\begin{sconjecture}[Abreu, Labbate, Sheehan \cite{ALS} ]\label{conj:ALS}
	The only snarks where all 2-factors consist of only odd cycles are the Petersen graph, the Blan\v usa-2 snark  and the Flower snarks. 
\end{sconjecture} 
This conjecture is false. 
\begin{observation}
	The smallest counterexample to Conjecture \ref{conj:ALS} has  26 vertices and is given in Appendix \ref{app:ALS}.
There is one additional counterexample on 34 vertices. The counterexamples can be obtained from \textit{House of Graphs}~\cite{HOG} by searching for the keywords ``counterexample * abreu''.
\end{observation}

The methods used to study cycle double covers with a given oddness can
also be applied to graphs with 2-factors with a small number of
components. With this in mind we have also investigated the smallest
and largest numbers of components for the small snarks. In Tables
\ref{tabmaxodd} and \ref{tabmaxcomp} the number of graphs with a given
maximum number of components and odd components in any 2-factor are
given.
\begin{table}
			\begin{tabular}{l|| r| r| r| r| r| }
				$n$ & 2 & 4 & 6 \\
				\hline
				10 &1&&\\
				18 & 2  & & \\
				20 & 6  & & \\
				22 & 20  & & \\
				24 & 34 & 4  & \\
				26 & 179&101  & \\
				28 & 1395& 1505  & \\
				30 & 10 039 &18 360 &  \\
				32 & 68 811& 224 164&84 \\
				34 &509 106 & 3 319 006& 5475 \\
				36 & 3 752 134 & 56 201 769 & 213 829 \\
			\end{tabular}
			\caption{The number of snarks with a given maximum number of odd components in any 2-factor.}
			\label{tabmaxodd}
\end{table}	

\begin{table}
			\begin{tabular}{l|| r| r| r| r| r| r| }
				$n$ & 2 & 3 & 4 & 5 & 6 & 7 \\
				\hline
				10 &1&&&&&\\
				18 & 1 & 1 & & & & \\
				20 & 1& 5& & & & \\
				22 & & 19& 1& & & \\
				24 & & 26& 12& & & \\
				26 & 1&81 &189 &9 & & \\
				28 & 1& 385&2306 &208 & & \\
				30 & &1012 &22 943 &4444 & & \\
				32 & & 1543& 204 463& 86 845 &208 & \\
				34 &1 & 2341& 1 957 813& 1 846 249& 27 183& \\
				36 & 1 & 3218 & 19 180 074 & 39 312 112 & 1 671 779 & 548 \\
			\end{tabular}
			\caption{The number of snarks with a given maximum number of components in any 2-factor.}
			\label{tabmaxcomp}
\end{table}

%-------------------------------------------------------------------------------	
\subsection{Automorphisms}
The Petersen graph is the only known snark with a vertex transitive automorphism group and it is also the smallest example of a vertex transitive graph which is not a Cayley graph. According to the so called Lovasz Conjecture there is only a finite, known, list of non-hamiltonian vertex transitive graphs. The name of the conjecture is a bit of a misnomer since the original problem posed by Lovasz \cite{Lo78} was to construct additional examples of vertex transitive graphs with no hamiltonian path. In stark contrast to this conjecture Babai has conjectured \cite{B95} that for any $\epsilon>0$ there should exist infinitely many vertex transitive graphs on $n$ vertices with no cycle of length greater than  $(1-\epsilon)n$. In connection with snarks the problem has been studied in  \cite{MR1863579}.

With these conflicting conjectures in mind it is natural to investigate the automorphism groups  of the small snarks.  In Table \ref{autom}	 we display the sizes of the automorphism group of all snarks on $n\leq 36$ vertices. As we can see, the Petersen graph has the largest automorphism group among the small snarks, and.we have verified that the Petersen graph is the only vertex transitive snark on $n\leq 36$ vertices.	We also conjecture that non-trivial automorphisms are uncommon among large snarks.
\begin{conjecture}
	The proportion of snarks $G$ on $n$ vertices with nontrivial automorphism group tends to 0 as $n\rightarrow \infty$.
\end{conjecture}

\begin{table}
	\begin{small}
			\begin{tabular}{l|| r| r| r| r| r| r| r| r| r| r| r|r|r|r|r|r|r}
				Order & 1 & 2 	& 3 	& 4 	& 6 	& 8 	& 12 & 16 & 20 & 24	& 28 & 32 & 36 & 48	& 64 & 80 & 120\\
				\hline
				10 & 	& 	& 	& 	& 	& 	& 	& 	& 	& 	& 	&	& 	& 	& 	& 	& 1 \\
				\hline
				18 & 	& 	& 	& 1	& 	& 1	& 	&  	& 	& 	& 	&	& 	& 	&  	&  	&\\
				\hline
				20 & 2	& 1	& 	& 2	& 	& 	& 	& 	& 1	& 	& 	&	& 	& 	& 	&  	& \\
				\hline
				22 & 4	& 11	& 	& 1	& 	& 1	& 2	& 1	& 	& 	& 	&	& 	& 	&  &  	&\\
				\hline
				24 & 21	& 9	& 	& 8	& 	& 	& 	& 	& 	& 	& 	&	&	& 	& 	&  	&\\
				\hline
				26 & 174	& 75	& 	& 23	& 	& 7	& 	& 1	& 	& 	& 	&	& 	& 	& 	&  	&\\
				\hline
				28 & 2536	& 290& 	& 62	& 1	& 6	& 2	& 2	& 	& 	& 1	&	& 	& 	&  	&  	&\\
				\hline
				30 & 26 214& 1924	& 	&226 	& 	&25 	& 	& 9	& 	& 	& 	&	& 	&  	& 	& 1 	&\\
				\hline
				32 & 278 718& 13 284& 	&973 	& 	&78 	& 	& 6	& 	& 	& 	&	& 	& 	&  	&  	&\\
				\hline
				34 & 3 684 637& 143 783	& 7	&4798& 7	&329& 1 	& 20	& 	& 3	& 	&1	& 	& 1	&  	&  	&\\
				\hline
				36 & 58 191 667 & 1 950 129 & 2	& 24 855	& 3	& 1044& 	3& 24& 	& 	& 	&2	& 1	& 1	&  1	&  	&\\
				\hline
			\end{tabular}
			\caption{The numbers of snarks with given orders of their automorphism group.  }\label{autom}
	\end{small}
\end{table}

%-------------------------------------------------------------------------------	
\subsection{Total chromatic number}
The \emph{total chromatic number} of a graph $G$ is denoted
$\chi''(G)$ and is defined as the minimum number of colours required
to colour $E(G)\cup V(G)$ such that adjacent vertices and edges have
different colours and no vertex has the same colour as its incident
edges. It is known \cite{Ro71} that for cubic graphs $\Delta+1\leq
\chi''(G)\leq \Delta+2$, and in \cite{CMRS} it was studied whether
there are snarks attaining the upper bound. We have extended their
tests:
\begin{observation}
			Every snark of order 36 or less and every weak snark of order 34 or less has total chromatic number 4.
\end{observation}

Among the colourable cyclically 4-edge-connected graphs there are
examples that have total chromatic number 5. However they are quite
rare as we can see in table \ref{tabletc}.

\begin{table}
	\begin{tabular}{l|| r r r r r r r r r r   }
		\hline
		$n$           & 8  & 10 & 12 & 14 & 16 & 18  & 20 & 22 & 24 & 26  \\ 
				\hline
		 $\chi''=5$& 1  &  3  &   5  &  2  & 3   & 102& 28 & 11 & 12 & 95 \\
		 		\hline
  $\lambda_c\geq4$&  2 &  5  & 18   & 84& 607 & 6100 & 78824 & 1195280 & 20297600 & 376940415
 	\end{tabular}
	\caption{The numbers of cyclically 4-edge-connected cubic graphs and the numbers of such graphs with $\chi''=5$.}
	\label{tabletc}
\end{table}

%-------------------------------------------------------------------------------	
\subsection{Strong snarks}
Apart from oddness, a number of other measures of how strongly non-colourable a snark is have been investigated, see e.g. \cite{MS10}. One of the first was the notion of strong snarks used by Jaeger in \cite{MR821502}. Following the notation from \cite{MR821502} we say that a snark $G$ is a \emph{strong snark} if for every edge $e$ in $G$, $G*e$ is uncolourable. Here $G*e$ denotes the unique cubic graph such that $G\setminus e$ is a subdivision of $G*e$.  Celmins \cite{celm} proved that a minimum counterexample to the cycle double cover conjecture must be a strong snark.

It is easy to see that if a snark contains a cycle that only misses
one vertex, then the snark cannot be strong. By checking all the
snarks with circumference at most $n-2$ we have found that:
\begin{observation}
	There are 7 strong snarks on 34 vertices and 25 strong snarks on 36 vertices. There are no strong snarks of smaller order. 
\end{observation}
In Figure \ref{strongsnark} we display the most symmetric strong snark
on 34 vertices. This and the other examples are presented in Appendix
\ref{app:strong}. They can also be obtained from \textit{House of
  Graphs}~\cite{HOG} by searching for the keywords ``strong snark''.

\begin{figure}[h]	
	\includegraphics[scale=0.45]{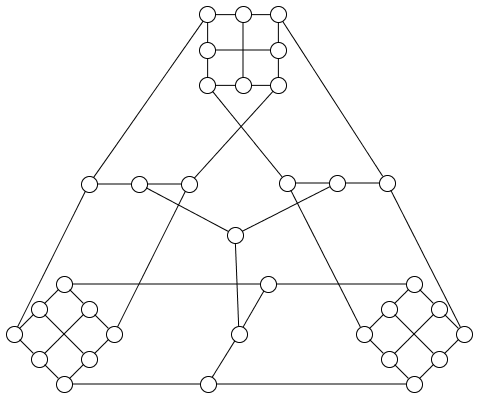}
		\caption{One of the seven strong snarks on 34 vertices.}
	\label{strongsnark}
\end{figure}

%--------------------------------------------------------------------------------------------------------------------------------------------------------------------------------------------
\section{Cycle double covers}\label{sec:cdc}
A \emph{cycle double cover} (usually abbreviated CDC) of a graph $G$ is a multiset of cycles such that each edge of $G$ lies in exactly two cycles. The Cycle Double Cover Conjecture is one of the most long-standing open problems in graph theory.
\begin{conjecture}[CDCC, Szekeres \cite{Sze73}, Seymour \cite{Sey79}]
	Every bridgeless graph has a cycle double cover.
	\label{cdcc}
\end{conjecture}			
The conjecture was first stated in print by Seymour \cite{Sey79} and
Szekeres \cite{Sze73} but at least some versions of the conjecture
seem to have been part of the folklore of the field even earlier. For
a survey see \cite{MR821502} or \cite{MR1426132}. The conjecture has
been proven for some large classes of graphs such as planar graphs and
4-edge-connected graphs. It is also well known that if the conjecture
holds for cubic graphs, then it is true for all graphs. If a cubic
graph is colourable, then it is easy to see that it has a CDC by
considering the cover given by each pair of colours from a given
proper colouring of the edges. It is also well known that a minimum
counterexample must have girth at least 5 and be cyclically
4-edge-connected (see e.g. \cite{MR821502}), and hence it is
sufficient to prove the CDCC for snarks.  These conditions have been
strengthened over the years and we now know that a minimum
counterexample must have girth at least 12, \cite{MR1743825}, and must
therefore be much larger than the snarks under investigation in this
paper.
	
However, there are various stronger versions of the CDCC, and auxiliary conjectures aimed at proving some version of the CDCC, which are not known to hold for small snarks. In the following sections we will investigate some of these variations and conjectures. In doing so, we will also  improve the result in \cite{hagglund10:2} regarding the truth of the CDCC for graphs with sufficiently long cycles. 

We say that a CDC $\mathcal D$ is a \emph{k-CDC} if the cycles of $\mathcal D$ can be coloured by $k$ colours in such a way that no pair of cycles with a common edge has the same colour. An \emph{even} graph is a graph where each vertex has even degree. Since a graph has a cycle decomposition into edge-disjoint cycles if and only if it is even, each colour class of cycles in a $k$-CDC is an even subgraph. It is often convenient to see a $k$-CDC as a $k$-multiset of even subgraphs. 
 	
%-------------------------------------------------------------------------------	
\subsection{The strong cycle double cover conjecture}
The following strengthening of the CDCC was formulated by Goddyn in \cite{goddynthesis}: 
\begin{conjecture}[Strong Cycle Double Cover Conjecture (SCDCC), Goddyn \cite{goddynthesis}]
	Let $G$ be a bridgeless graph. Then for every cycle $C$ in $G$ there is a CDC that contains $C$.
\end{conjecture}
It is easy to see that this holds for colourable cubic graphs, where in fact any 2-regular subgraph can be extended to a CDC. 
\begin{proposition} \label{prop:cdc}
	Let $G$ be a colourable cubic graph and let $D$ be any 2-regular subgraph of $G$. Then $G$ has a 4-CDC where $D$ is one of the colour classes.
\end{proposition}
\begin{proof}
	Let $\{C_1,C_2,C_3\}$ be the complementary 2-factor to the 1-factors that corresponds to the 3-edge-colouring. Then $\{D,D\Delta C_1,D\Delta C_2,D\Delta C_3\}$ is a 4-CDC of $G$. Here the symmetric difference is taken with the cycles seen as sets of edges.
\end{proof}
Using standard reductions it is quite easy to see that a minimal counterexample to the SCDCC must be a snark (see e.g. \cite{MR1825619} for details). 
	
There are not many classes of snarks where the SCDCC is known to be true. However one such class is the class of hypohamiltonian snarks, for which Fleischner and H\"aggkvist have shown the following theorem. 
\begin{theorem}[Fleischner, H\"aggkvist \cite{FH09}]\label{hypothm}
		Let $\mathcal C$ be a set of disjoint cycles in a hypohamiltonian graph $G$ such that $\cup_{C\in\mathcal C} V(C) \neq V(G)$. Then $G$ has a CDC $\mathcal D$ such that $\mathcal C\subset \mathcal D$. 
\end{theorem}
In the same paper they also show that any cycle in a cubic graph  that misses at most one vertex is part of a CDC.
\begin{theorem}[Fleischner, H\"aggkvist \cite{FH09}]
	Let $G$ be a cubic graph on $n$ vertices and let $C$ be a cycle in $G$. If $|V(C)|\geq n-1$ then $C$ can be extended to a CDC. 
	\label{haggkvist1}
\end{theorem}
Note that Theorem \ref{hypothm} cannot be generalized to general snarks. Many non-hypohamiltonian snarks have 2-regular subgraphs that miss at least one vertex that cannot be extended to a CDC, but of course these 2-regular subgraphs have several components. 
	
\begin{table}
	\begin{tabular}{l|| cccp{2.7cm}p{2.5cm}}
	Order &  Total & $\lambda_c =5$ & Shortest  & Maximum \\
	\hline
	10-20 	 	& 0	&   -  & -& -\\
	22  		& 2	& 0	& 20	& 1 \\
	24  		& 1	& 0	& 22 & 2 \\
	26  		& 7	& 0	& 23 & 4 \\
	28  		& 25	& 0	& 24 & 10 \\
	30  		& 228 & 0  & 26 & 24 \\
	32   		& 1456 &4 & 26 & 48 \\
	34		& 6902 &33	& 24 & 104 \\
	36             &  35 316&202 & 24 & 180 \\
	\end{tabular}
	\caption{The number of small stable snarks,  the length of the shortest stable cycle among the stable snarks of a given order, and the maximum number of stable cycles in any snark of that order.
}
	\label{stabletab}		
\end{table}

The full set of cycles in cubic graphs is very large and in order to
simplify the study of the SCDCC we may without loss of generality
consider only a certain subset of cycles. Let $G$ be a bridgeless
cubic graph with a cycle $C$. 
We say that $C$ is \emph{stable} if there is no cycle $C'$ distinct from $C$
such that $V(C)\subseteq V(C')$.
A graph that contains stable cycles is called a \emph{stable
  graph}. Stable cycles were studied in
e.g. \cite{MR2305899} and \cite{MR1825619} and the following
observation shows the connection with the SCDCC.
	
\begin{proposition}
	Let $G$ be a minimal counterexample to the SCDCC and $C$ a
        cycle which cannot be extended to a CDC of $H$. Then $G$ is a
        snark and $C$ is a stable cycle.
\end{proposition}
	
Fleischner gave the first construction of a cubic graph with stable
cycles \cite{MR1283310}. However, his construction does not give rise
to snarks since all resulting graphs have cyclic 3-edge cuts. The
smallest cubic colourable non-hamiltonian graph (the graph can be
found in \cite{MR2122107}) is another example of a stable graph, which
also happens to be cyclically 4-edge connected, but is of course not a
snark. The first example of a snark with stable cycles was given by
Kochol in \cite{MR1825619} where he gives a construction that yields
an infinite family of stable snarks.
	
By computer search we tested all snarks on 36 vertices and less for the presence of stable cycles. The number of stable snarks of a given order can be found in Table \ref{stabletab}, where also the lengths of the shortest stable cycles of a given order and the maximum number of stable cycles in any single graph is presented. In \cite{hagglund10:2} it was shown that there is an infinite family of snarks with a stable cycle of constant length. We have also verified that each of the stable cycles in the snarks of order 36 or less is part of some CDC. 
\begin{observation}
	Every stable cycle in the snarks of order at most 36 can be included in some CDC. 
\end{observation}
\begin{corollary}\label{scdcc32}
	There are no counterexamples to the strong cycle double cover conjecture of order $n\leq 36$. 
\end{corollary}
Although we have verified the SCDCC for all snarks on 36 vertices and less, we have also observed that there are cycles which can only be part of 
exactly one CDC. Due to the increasing numbers of both cycles and snarks this is a time consuming property to test so here we only give data up to $n=26$
\begin{observation}
	There is one snark of order 10, one snark of order 18, there
        are 2 snarks of order 20, 6 of order 22, 9 of order 24, and 42
        of order 26 which have cycles which are part of exactly one
        CDC.
\end{observation}

As shown in \cite{hagglund10:2} lower bounds on the size of a minimal counterexample to the SCDCC have implications for the existence of CDCs in graphs with high circumference. 
\begin{theorem}[H\"agglund, Markstr\"om \cite{hagglund10:2}]\label{thmcirc}
		If  $k\in\mathbb N$ and the SCDCC holds for all cubic graphs of order at least $4k$  then every cubic graph of order $n$ with circumference at least $n-k-1$ has a CDC. 
\end{theorem}
See \cite{hagglund10:2} for a proof. 
Together with corollary \ref{scdcc32} this gives the following:
\begin{corollary}\label{circ}
	Let $G$ be a bridgeless cubic graph of order $n$. If $G$ has a cycle of length at least $n-10$ then $G$ has a CDC. \label{thm1}
\end{corollary}

%-------------------------------------------------------------------------------			
\subsection{Semiextensions and superstable cycles}
It was shown by Moreno and Jensen in \cite{MR2305899} that a minimum counterexample to the SCDCC must also satisfy a stronger condition than having stable cycles. Let $C$ be a cycle in a graph $G$. We say that $C$ has a \emph{semiextension} if there exists another cycle $D\neq C$ in $G$ such that for every path $P=xv_1v_2\dots v_k y$ where $x\in V(C)\setminus V(D)$, $y\in V(C)\cup V(D)$ and $v_1,\dots,v_k\in V(G)\setminus(V(C)\cup V(D))$ we have another $x-y$-path $P'$ where $P'\subset C\Delta D$. A cycle with no semiextension is called a \emph{superstable cycle}. 
\begin{theorem}[Esteva and Jensen \cite{MR2305899}]
		Let $(G,C)$ be a minimal counterexample to SCDCC. Then $C$ must be a superstable cycle. 
\end{theorem}
They also gave the following conjecture, which implies the SCDCC.
\begin{conjecture}[Esteva and Jensen \cite{MR2305899}]
		There are no 2-connected cubic graphs with superstable cycles. 
\end{conjecture}
Using a computer we have verified this conjecture for all snarks of order at most 36. 
\begin{observation}
		All cycles in the snarks on $n\leq 36$ vertices have semiextensions. 
\end{observation}

A \emph{strong semiextension} of a cycle $C$ is a semiextension where $C\Delta D$ has a single component. 
In \cite{MR2531643} Moreno and Jensen posed the following problem:
\begin{problem}[Esteva and Jensen \cite{MR2531643}]
		Are there any 3-edge-connected graphs which have a cycle which does not have a strong semiextension? 
\end{problem}
We have tested this property for all small snarks.
\begin{observation}
	All cycles in snarks on $n\leq 36$ vertices have a strong semiextension.
\end{observation}
	
%-------------------------------------------------------------------------------	
\subsection{Even Cycle Double Covers}
In \cite{markstrom10} Markstr\"om investigated cycle double covers restricted to using only even length cycles and posed the following problem:
\begin{problem}[Markstr\"om \cite{markstrom10}]
	Is the Petersen graph the only cyclically 4-edge connected
        cubic graph which does not have a CDC consisting
        only of cycles of even length?
\end{problem}
Cycle double covers of this type are called even cycle double covers. 
In \cite{markstrom10} it is shown that a possible smallest example of another graph than the Petersen graph
would be a snark. We have verified this property for all snarks on 36 vertices and less.
	
\begin{corollary}\label{ecdcc}
	All cyclically 4-edge connected cubic graphs on  $n\leq 36$  vertices, except the Petersen graph, have an even cycle double cover.
\end{corollary}

Note that this does not hold in general if the graph has cyclic 3-edge-cuts. The smallest example of a 3-connected cubic graph, except the Petersen graph, without a CDC with only even cycles has 12 vertices.  It is also easy to show that there are cubic graphs in which not all even cycles are part of an even cycle double cover. Hence there is no "strong" version of this problem.

%-------------------------------------------------------------------------------	
\subsection{Extension of 2-regular subgraphs to CDCs} \label{sec:counterex}
The following question is a natural generalization of the SCDCC: \emph{Can every disjoint union of cycles be extended to a CDC?} It is easy to see that this is not the case. Consider for instance the Petersen graph where no 2-factor can be extended to a CDC. However for colourable cubic graphs this is true by Proposition \ref{prop:cdc}.

The following was posed as a problem in \cite{MR1239236} and as a conjecture in \cite{MR1426132}. 
\begin{sconjecture}[Jackson \cite{MR1239236, MR1426132}]\label{conj:ex1}
		Let $G$ be a cyclically 5-edge-connected cubic graph and $\mathcal D$ be a set of pairwise disjoint cycles of $G$. Then 
		$\mathcal D$ is a subset of a CDC, unless $G$ is the Petersen graph. 
		\label{conj:jackson}
\end{sconjecture}
We have found 12 counterexamples on 34 vertices to this conjecture. One of them is shown in Figure \ref{cyc5perm}, where the set of disjoint cycles that cannot be extended to a CDC is the bold 2-factor. An interesting observation is that these are exactly the same graphs that we found to be counterexamples to Conjecture \ref{conj:zhang}.
\begin{observation} \label{obs:counterexample}
		There are exactly 12 cyclically 5-edge-connected snarks on 34 vertices with a 2-factor that cannot be part of any CDC and the only smaller 
		graph with that property is the Petersen graph. These 12 graphs and the Petersen graph are also permutation graphs and the 2-factors that 
		cannot be part of any CDC have exactly two induced components. There are also 44 cyclically 5-edge-connected snarks on 36 vertices with 
		2-regular subgraphs, not all of which are 2-factors, that cannot be part of any CDC. 
\end{observation}
We also have a complementary observation.
\begin{observation} \label{obs:pos2f}
		All snarks on $n\leq 36$ vertices, except the Petersen graph and the graph in Figure \ref{strongsnark}, have a 2-factor which is part of a CDC. 
\end{observation}

The following weaker form of Jackson's conjecture was proposed by Zhang, and is also refuted by Observation \ref{obs:counterexample}. 
\begin{sconjecture}[Zhang \cite{MR1426132}]\label{conj:ex2}
		Let $G$ be a permutation graph with the chordless cycles $C_1$ and $C_2$ where $C_1\cup C_2$ is a 2-factor. If $G$ is cyclically 
		5-edge-connected and there is no CDC which includes both $C_1$ and $C_2$, then $G$ must be the Petersen graph. 
\end{sconjecture}
There are also more specific versions of Conjecture \ref{conj:ex1}, see \cite{MR988642}, which can be reduced to the case of $G$ being a permutation snark with no removable cycles, see \cite{AGZ}. However by Observation \ref{obs:rem} all small permutation snarks, except the Petersen graph, have removable cycles and hence cannot be counterexamples to  these conjectures.

An eulerian $(1,2)$-weight of a cubic graph is a map $w:E(G)\rightarrow \{1.2\}$ such that the sum of the weights on any edge-cut is even. Given a weight $w$ a set of cycles from $G$ is said to be a compatible cycle cover of  $G$ if each edge $e$ belongs to exactly $w(e)$ of the cycles. In particular, a cycle double cover corresponds to a compatible cover for the constant weight $w=2$. The following conjecture for cubic graphs  was stated in \cite{MR2644230} as a version for cubic graphs of three older conjectures:
\begin{sconjecture}[Zhang \cite{MR2644230}]\label{FGJ3}
	Let $G$ be a bridgeless cubic graph and $w$ be an eulerian $(1,2)$-weight of $G$ where $\sum_{e\in T}w(e)>4$ for every cyclic 
	edge-cut $T$. Then if $(G,w)$ has no compatible cycle cover, then $G$ must be the Petersen graph. 
\end{sconjecture}
The older conjectures concern compatible cycle decompositions of
eulerian graphs, an area closely tied to the study of cycle double
covers. Since our emphasis is on Conjecture \ref{FGJ3}, which we will
refute, we will refer to the original sources for the definitions of
some terms used in this section. A \emph{transition system} $T$ is a
set of paths of length 2 in an eulerian graph $G$, such that the set
of paths with a given vertex $v$ as their midpoint induce a partition
of the edges incident with $v$. Let $T(v)$ denote the set of paths in
$T$ with midpoint $v$.  A cycle decomposition of $G$ is said to be
compatible with $T$ if no path from $T$ is a subgraph of a cycle in
the decomposition. In \cite{MR2355128} Fleischner, Genest and Jackson
showed that the following two conjectures are equivalent.

\begin{sconjecture}[Fleischner, Genest and Jackson \cite{MR2355128}]\label{FGJ1}
	 	Let $G$ be a connected 4-regular multigraph with loops
                and $T$ be a transition system. Then $(G,T)$ has no
                compatible cycle decomposition if and only if
                $(G,T)\in\mathcal R$.
\end{sconjecture}

\begin{sconjecture}[Fleischner, Genest and Jackson \cite{MR2355128}]\label{FGJ2}
		If $G$ is an essentially 6-edge-connected 4-regular
                multigraph with loops with a transition system $T$,
                then $(G,T)$ has no compatible cycle decomposition if
                and only if $(G,T)$ is the bad loop or the bad $K_5$.
\end{sconjecture}

Here a  multigraph with loops is said to be \emph{essentially 6-edge-connected} if
every edge cut of size less than 6 is the set of edges incident to a
vertex. The bad loop and the bad $K_5$ are two specific transition
systems. We refer the reader to \cite{MR2355128} for their definition
and the definition of the family $\mathcal R$.  Furthermore, in
\cite{Ge10} the following conjecture was proposed, and was shown to be
equivalent to Conjecture \ref{FGJ2}. We refer the reader to
\cite{Ge10} for the definitions of the involved objects.
\begin{sconjecture}[Genest \cite{Ge10}]\label{conj:G}
	The only Sabidussi orbits that are pure, prime and circle orbits are the orbit of the white isolated vertex and the orbit of the white pentagon.
\end{sconjecture} 

Here we provide a proof of the equivalence of Conjecture \ref{FGJ3} and the three older conjectures restricted to simple graphs.
\begin{proposition}
		Conjecture \ref{FGJ3} and Conjecture \ref{FGJ2} (restricted to simple graphs) are equivalent. 
\end{proposition}

\begin{proof}
	Assume that G is a 4-regular, essential 6-edge-connected graph
        and T is a transition system of G. Let $v$ be a vertex of $G$
        and assume that $T(v) = \{e_1e_2,e_3e_4\}$.  Split $v$ into
        two new vertices $v_1'$ and $v_2'$ and add the edge
        $v_1'v_2'$. Let $w(e_i) = 1$ for $i=1,\dots, 4$ and
        $w(v_1'v_2')=2$. Repeat this for all vertices in $G$ and
        denote the resulting graph $G'$. Then it is easy to see that
        $G'$ is a cubic graph with an eulerian weight $w$ and with
        $\sum_{e\in L}w(e)\geq 6$ for any non-trivial edge cut $L$. A
        compatible cycle cover for $(G',w)$ now corresponds to a
        compatible cycle decomposition in $G$.
		
	Conversely, assume that $H$ is a cubic graph with an eulerian
        weight $w$ where the weigh sum across any cyclic edge cut it
        at least $6$ (it cannot be 5 since it would then contradict
        the fact that $w$ is an eulerian weight). Every edge in $H$
        has either one or three adjacent edges with weight 2 since $w$
        is eulerian. If $v$ is a vertex in $H$ with neighbours
        $v_1,v_2,v_3$ and $w(vv_1) = w(vv_2) = w(vv_3) = 2$, then we
        replace $v$ with a triangle with vertices $u_1,u_2,u_3$ where
        all edges have weight 1 and add the edges
        $u_1v_1,u_2v_2,u_3v_3$ with weight 2. We repeat this until
        there are no vertex where all adjacent edges have weight
        2. Denote the new graph by $H'$. The edges of weight 2 now
        form a perfect matching in $H'$. By contracting all the weight
        2 edges in $H'$ we get a 4-regular graph $H''$. Let $T$ we the
        transition system that corresponds to the weight 1 edges
        adjacent to each edge. Obviously $H''$ cannot have any
        non-trivial edge cut of cardinality less then 6 by the
        condition on $w$. It is now easy to see that $(H'',T)$ has a
        compatible cycle decomposition if and only if $(H,w)$ has a
        compatible cycle cover.
\end{proof}
Consider the graph in Figure \ref{cyc5perm} and let $w(e)=1$ if $e$ is part of the bold 2-factor and $w(e)=2$ otherwise.  By Observation \ref{obs:counterexample} this 2-factor cannot be part of any CDC and the weight sum over any cyclic edge-cut must be strictly greater than 4 since the graph is cyclically 5-edge-connected. Therefore it is a counterexample to the above conjectures. 
\begin{observation} \label{obs:counterexample3}
		The graph in Figure \ref{cyc5perm}, and all graphs in
                Appendix \ref{app:ex1}, are counterexamples to
                Conjecture \ref{FGJ3}. Hence Conjectures \ref{FGJ1},
                \ref{FGJ2} and \ref{conj:G} are false as well.
\end{observation}

%--------------------------------------------------------------------------------------------------------------------------------------------------------------------------------------------
\section{Oriented cycle double covers and k-CDCs}	
In Section \ref{sec:cdc} we defined the notion of a $k$-CDC. The
chromatic number of a CDC is a well studied topic. By an early result
of Tutte \cite{tu49} we know that a cubic graph has a 3- or 4-CDC if
and only if it is colourable. Celmins \cite{celm}
conjectured that there exists an integer $k$ such that every
bridgeless cubic graph has a $k$-CDC, and that in fact $k=5$. A 
minimal counterexample would be a snark. This is the
5-cycle double cover conjecture, or 5-CDCC,
\begin{conjecture}[Celmins \cite{celm}]\label{cdcc:k}
	Every bridgeless cubic graph has a 5-CDC. 
\end{conjecture} 

Another strengthening of the CDCC comes from considering orientations of the cycle, which can be interpreted in terms of orientable surface 
embeddings of the graph. An \emph{orientable CDC} is a CDC where we can put an orientation on the cycles in such a way that for every edge 
we have that the two cycles that cover that edge must have different orientations. The strongest conjecture in this context has been the following.
\begin{conjecture}[Archdeacon \cite{MR754919} and Jaeger \cite{Jae:80}]\label{conj:o5cdcc}
	Every bridgeless graph has an orientable 5-CDC. 
\end{conjecture}
By a computer search we have verified that this conjecture, and hence also Conjecture \ref{cdcc:k}, holds for all snarks on 36 vertices and less. 
\begin{observation}\label{obs:o5}
	Every snark on $n\leq 36$ vertices has an orientable 5-CDC. 
\end{observation}

%-------------------------------------------------------------------------------	
\subsection{The Strong CDC conjecture and $k$-CDCs}\label{s5}
In a common refinement of the strong CDCC and the 5-CDCC one may ask whether any cycle is part of some 5-CDC. In \cite{arthur1} Hoffmann-Ostenhof gives the following characterization of 2-regular subgraphs which can be part of a 5-CDC:
\begin{theorem}[Hoffmann-Ostenhof \cite{arthur1}]
	Let $G$ be a cubic graph and $D$ be a 2-regular subgraph of $G$. Then there is a 5-CDC which contains $D$ if and only if $G$ contains 
	two 2-regular subgraphs $D_1$ and $D_2$ such that $D\subset D_1$, $M=E(D_1)\cap E(D_2)$ is a matching and $G-M$ has a nowhere-zero 4-flow. 
\end{theorem}
In the same paper he conjectured that not only can every cycle be included in some CDC but it can in fact always be part of some 5-CDC. 
\begin{conjecture}[Hoffmann-Ostenhof \cite{arthur1}]\label{arthur1}
	Every cycle $C$ in a bridgeless cubic graph $G$ is contained in some 5-CDC of $G$. 
\end{conjecture}
We have verified this conjecture for all snarks of order 34 and less. The test for this property is extremely time consuming and with our available programs and computers we could not perform the  test for all snarks on 36 vertices, but the subset of cyclically 5-edge connected snarks was manageable.
\begin{observation}
	Every cycle in a bridgeless cubic graph with $n\leq 34$ vertices is part of a 5-CDC. This is also true for  all cyclically 5-edge-connected snarks on 36 vertices.
\end{observation}
For colourable cubic graphs we have already shown that the statement
in Conjecture \ref{arthur1} can be simultaneously strengthened to give
a 4-CDC and allow $C$ to be any 2-regular subgraph, rather than just a
cycle.  For snarks we cannot find 4-CDCs but it is not immediately
obvious that $C$ cannot have several components, if we exclude the
Petersen graph. However by Observation \ref{obs:counterexample} we
know that this generalization is false, even if $5$ is replaced by a
larger value and $G$ is required to be cyclically 5-edge
connected. Among the cyclically 4-edge connected snarks this fails
already for $n=18$, since both snarks of that order have 2-factors
which cannot be included in a CDC.  However, for sufficiently small
snarks there is either an including 5-CDC or no including CDC at all.
\begin{observation}\label{obs:noext}
	Given a 2-regular subgraph $\mathcal{D}$ of a cyclically 4-edge connected snark on $n\leq 28$ vertices, there is either no CDC which contains 
	$\mathcal{D}$  or there is a 5-CDC which contains $\mathcal{D}$. There are exactly 68 snarks on $30$ vertices  with 2-regular subgraphs which are contained in a CDC but not in a 5-CDC.
\end{observation}
In Figure \ref{fig:no5cdc} we display one of the 30 vertex snarks mentioned in Observation~\ref{obs:noext}. The full set is presented in Appendix \ref{app:no5cdc}.
\begin{figure}[h!]	
	\includegraphics[scale=0.5]{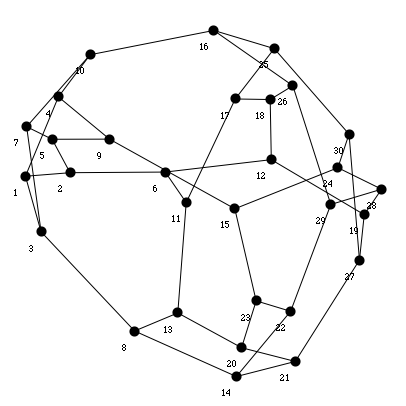}
		\caption{A snark where the 2-factor given by the cycles $[1, 4, 9, 5, 2]$ and [3, 8, 14, 22, 23, 15, 24, 30, 25, 17, 18, 12, 6, 11, 13, 20, 21, 27, 19, 28, 29, 26, 16, 10, 7]  cannot be part of any 5-CDC.}\label{fig:no5cdc}
\end{figure}

For cyclically 5-edge connected snarks the only exception is the Petersen graph.
\begin{observation}\label{obs:noext2}
	Every 2-regular subgraph of a cyclically 5-edge connected snark on $n\leq 32$ vertices, except the Petersen graph, is part of a 5-CDC.  Among the cyclically 5-edge connected snarks on $34$ vertices the only exceptions are the 12 snarks from Observation \ref{obs:counterexample} with their specified 2-factor.
	
	On $36$ vertices there are 44 cyclically 5-edge connected snarks which have 2-regular subgraphs which are  not part of any 5-CDC, but are part of some CDC.  Some of these  2-regular subgraphs are not 2-factors.
\end{observation}
The snarks on $36$ vertices mentioned in the observation are presented in Appendix \ref{app:no5cdc2}. Without access to the snarks from Observation \ref{obs:counterexample} one might have been tempted to conjecture that the only exception was the Petersen graph, again demonstrating the importance of computer generated examples as an aid for our understanding.  
\begin{figure}[h!]
	\includegraphics[scale=0.50]{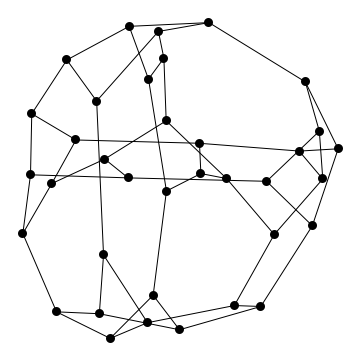}
		\caption{One of the 36 vertex graphs mentioned in Observation \ref{obs:noext2}}
\end{figure}
	
By using Observation \ref{obs:o5} and the methods from  \cite{hagglund10:2}  we can deduce a strengthening of Corollary \ref{circ} for slightly larger values of the girth.
\begin{theorem}
	Let $G$ be a bridgeless cubic graph. If $G$ has a cycle of length at least $n-8$ then $G$ has a 6-CDC. 
\end{theorem}

%--------------------------------------------------------------------------------------------------------------------------------------------------------------------------------------------
\section{Shortest cycle covers}	
The results from the previous sections can be combined in order to provide length bounds on cycle covers. Recall that a \emph{cycle cover} $\mathcal{F}$ of a graph $G$ is a set of cycles from $G$ such that every edge in $G$ belongs to at least one of the cycles in $\mathcal{F}$. The length of $\mathcal{F}$ is the sum of the lengths of all cycles in $\mathcal{F}$. The shortest length of a cycle cover is a well studied property and in \cite{AT} Alon and Tarsi made the following conjecture:
\begin{conjecture}[Alon, Tarsi \cite{AT}]\label{conj:AT}
	If $G$ is a bridgeless graph with $m$ edges, then the length of the shortest cycle cover of $G$ is at most $\frac{7m}{5}$.
\end{conjecture}
The minimum length among the cycle covers of a cubic graph can be bounded by using the following simple lemma.
\begin{lemma}\label{lem:short}
	If a cubic graph $G$ has a 2-regular subgraph $C$ and a CDC $\mathcal{D}$ such that $C\in\mathcal{D}$, then the length of the shortest cycle cover of $G$  is at most $2m-|C|$, where $m$ is the number of edges in $G$
\end{lemma}
For a colourable cubic graph this immediately implies that the shortest cycle cover has length at most $\frac{4m}{3}$, by taking $C$ to be a 2-factor and applying Proposition \ref{prop:cdc}. By using Corollary \ref{scdcc32} with a longest cycle $C$, and the lower bound on the circumference from Section \ref{sec:circum}, we now have:
\begin{corollary}\label{cor:short}
	Every snark on $n\leq 36$ vertices, except the Petersen graph and the graph in Figure \ref{strongsnark}, has a cycle cover of length $\frac{4m}{3}$ and the two exceptional graphs have  shortest cycle covers of length $\frac{4m}{3}+1$.  Hence there are no counterexamples to  Conjecture \ref{conj:AT} among the snarks on $n\leq 36$ vertices.
\end{corollary}
\begin{proof}
	By Observation \ref{obs:pos2f} we know that all snarks except the two exceptional ones have 2-factors which are part of some CDC, and by Lemma \ref{lem:short}  they have a cycle cover of length $\frac{4m}{3}$.  Both of the exceptional graphs have cycles of length $n-1$ and by Corollary \ref{scdcc32} those cycles are part of some CDC, and the lemma gives a cycle cover of length $\frac{4m}{3}+1$. A separate computer test shows that for the exceptional graphs there are no shorter cycle covers. 
\end{proof}

We believe that results of Corollary \ref{cor:short} are typical for cyclically well connected cubic graphs and give to following refinement of Conjecture \ref{conj:AT}
\begin{conjecture}\label{conj:short}
	If $G$ is a snark with $m$ edges, then the length of the shortest cycle cover of $G$ is at most $\frac{4m}{3}+o(m)$.
\end{conjecture}
It is known that for cubic graphs of cyclic connectivity 3  the constant $\frac{7}{5}$ in Conjecture \ref{conj:AT} would be optimal.

%--------------------------------------------------------------------------------------------------------------------------------------------------------------------------------------------
\section{Jaeger's flow and colouring conjectures}	
In \cite{Jae:80} Jaeger studied nowhere zero flows on graphs with the flow values restricted to certain subsets of a group. These questions are related to the cycle double cover and the various flow conjectures of Tutte. See \cite{Jae:88} for a survey by the same author. In \cite{Jae:80} Jaeger studied a family of partial orders on  the properties of flow spaces of  graphs and made a conjecture about one such partial order which implies the 5-CDCC, Conjecture  \ref{conj:o5cdcc}. We will state two of the equivalent forms of this conjecture.  

Let $\omega_G(v)$ denote the set of edges in the graph which are incident to the vertex $v$. Let $T(G)=\{\omega_G(v)|v\in V(G)\}$.  Given two cubic graphs $G$ and $H$, an \emph{$H$-colouring} of $G$ is a mapping $f$  from $E(G)$ to $E(H)$ such that $T(G)$ is mapped to $T(H)$.  Let $P$ denote the Petersen graph. In this terminology Jaeger conjectured that:
\begin{conjecture}[The Petersen colouring conjecture, Jaeger \cite{Jae:80}]\label{conj:pcol}
	Every bridgeless cubic graph has a $P$-colouring. 
\end{conjecture}
A colourable cubic graph trivially satisfies this conjecture by letting $f$ be a mapping onto the three edges incident to a single vertex in $P$. It is also easy to show that a minimal counterexample cannot have a nontrivial 3-edge cut, and hence no triangles. Thus a minimal counterexample must be a weak snark.

The second equivalent form of this conjecture \cite{Jae:85} concerns
another type of edge-colourings.  Let $c:E(G)\rightarrow
\{1,2,3,4,5\}$ be a proper edge colouring of $G$. An edge $e=\{u,v\}$
is poor if $|c(\omega_G(u)\cup \omega_G(v))|=3$ and rich if
$|c(\omega_G(u)\cup \omega_G(v))|=5$. A proper colouring in which all
edges are poor is in fact a proper 3-edge colouring, if $G$ is
connected. A colouring in which every edge is rich is called a strong
edge colouring. Such colourings were first studied in \cite{FJ:83}, in terms of
induced matchings. A colouring $c$ is called \emph{normal} if all
edges are either rich or poor. In \cite{Jae:85} it was shown that
Conjecture \ref{conj:pcol} is equivalent to:
\begin{conjecture}[Jaeger \cite{Jae:85}]\label{conj:pcol2}
	Every bridgeless cubic graph has a normal 5-edge colouring. 
\end{conjecture}
We have tested these conjectures for the small snarks and weak snarks.
\begin{observation}\label{obs:pcol}
	There are no counterexamples to Conjecture \ref{conj:pcol} on
        $n\leq 34$ vertices, and none of the snarks on $n=36$ vertices
        is a counterexample either. The only snark with $n\leq 36$
        vertices which has a normal 5-edge colouring with only rich
        edges is the Petersen graph.
\end{observation}

As mentioned earlier, Jaeger's conjectures imply the 5-CDCC and so Observation \ref{obs:pcol} provides a partial verification of Observation \ref{obs:o5}. Another related conjecture is: 
\begin{conjecture}[Fulkerson's conjecture \cite{Fu:71}]\label{conj:fu}
	Every bridgeless cubic graph has six perfect matchings such that any edge is a member of exactly two of them.
\end{conjecture}
The conjecture is trivially true for colourable cubic graphs, and it is not difficult to show that a minimum counterexample must be a snark. As noted by Jaeger \cite{Jae:85}, any graph with a $P$-colouring has six perfect matchings satisfying the condition in Conjecture \ref{conj:fu}. So we have the following corollary 
\begin{corollary}
	There are no counterexamples to Conjecture \ref{conj:fu} on  $n\leq 36$ vertices.
\end{corollary}
Berge proposed a seemingly weaker form of Fulkerson's conjecture:
\begin{conjecture}[Berge]
	Given a 2-connected cubic graph $G$. There exist five perfect matchings in $G$ such that every edge is part of at least one of them.
\end{conjecture}
However, as proven in \cite{Maz10}, this conjecture is in fact equivalent to that of Fulkerson, and by our earlier observation true for graphs with $n\leq 36$ vertices.

%--------------------------------------------------------------------------------------------------------------------------------------------------------------------------------------------	
\section*{Acknowledgments}
This research was conducted using the resources of High Performance Computing Center North (HPC2N), C3SE, PDC, NSC, and   Swegrid in Sweden, and the Stevin supercomputer Infrastructure at Ghent University. 

Jan Goedgebeur is supported by a PhD grant of the Research Foundation of Flanders (FWO). Jonas H\"agglund is supported by the National Graduate School in Scientific Computing (NGSSC).

%--------------------------------------------------------------------------------------------------------------------------------------------------------------------------------------------	
%\bibliographystyle{amsplain}
%\bibliography{Numbers,literatur}
\providecommand{\bysame}{\leavevmode\hbox to3em{\hrulefill}\thinspace}
\providecommand{\MR}{\relax\ifhmode\unskip\space\fi MR }
% \MRhref is called by the amsart/book/proc definition of \MR.
\providecommand{\MRhref}[2]{%
  \href{http://www.ams.org/mathscinet-getitem?mr=#1}{#2}
}
\providecommand{\href}[2]{#2}

%--------------------------------------------------------------------------------------------------------------------------------------------------------------------------------------------	

\section*{Appendices}
The graphs in these appendices are given in the following format: Each list corresponds to an adjacency list for the graph where only higher numbered neighbours are listed. That is, first comes the neighbours of vertex 1, next the higher numbered neighbours of vertex 2 and so on.

All material in the appendices is also available as computer readable text in the Arxiv preprint version of this paper with identifier: arXiv:1206.6690.

%--------------------------------------------------------------------------------------------------------------------------------------------------------------------------------------------
\subsection{A cubic graph where all 2-factors only consist of odd cycles}\label{app:ALS}
{\ \\}This graph is the smallest counterexample to Conjecture \ref{conj:ALS}:\\
\begin{tiny}
\{2, 3, 4, 5, 6, 7, 8, 9, 10, 7, 9, 8, 10, 11, 12, 13, 14, 15, 16, 15, 17, 18, 19, 18, 20, 18, 21, 
 22, 21, 23, 23, 24, 22, 24, 25, 26, 26, 25, 26\}
 \end{tiny}
%--------------------------------------------------------------------------------------------------------------------------------------------------------------------------------------------
\subsection{Cyclically 5-connected snarks with stable cycles}\label{app:stab}
{\ \\}The four cyclically 5-connected stable snarks on 32 vertices:\\
\begin{tiny}
\{2, 3, 4, 5, 6, 7, 8, 9, 10, 7, 9, 11, 12, 13, 11, 14, 14, 12, 15, 16, 17, 18, 19, 20, 21, 22, 23, 
24, 25, 26, 21, 27, 22, 28, 29, 30, 28, 31, 25, 29, 27, 30, 30, 28, 32, 31, 32, 32\}\\
\{2, 3, 4, 5, 6, 7, 8, 9, 10, 7, 9, 11, 12, 13, 11, 14, 14, 12, 15, 16, 17, 18, 19, 20, 21, 22, 23, 
24, 25, 26, 21, 27, 22, 28, 29, 30, 28, 31, 25, 29, 30, 31, 30, 28, 32, 31, 32, 32\}\\
\{2, 3, 4, 5, 6, 7, 8, 9, 10, 7, 9, 11, 12, 13, 11, 14, 14, 15, 16, 17, 13, 18, 19, 20, 21, 22, 23,
 24, 25, 26, 27, 28, 21, 23, 29, 30, 24, 25, 31, 31, 27, 29, 28, 30, 32, 29, 32, 32\}\\
\{2, 3, 4, 5, 6, 7, 8, 9, 10, 7, 9, 11, 12, 13, 11, 14, 14, 15, 16, 17, 13, 18, 19, 20, 21, 22, 23, 
24, 25, 26, 27, 28, 21, 23, 29, 30, 24, 31, 32, 31, 27, 28, 29, 30, 31, 32, 30, 32\}
\end{tiny}

%--------------------------------------------------------------------------------------------------------------------------------------------------------------------------------------------
\subsection{Strong snarks}\label{app:strong}
{\ \\}The 7 strong snarks on 34 vertices:\\
\begin{tiny}
\{5, 9, 13, 12, 15, 25, 5, 8, 10, 6, 8, 9, 6, 15, 8, 14, 19, 10, 12, 20, 26, 29, 27, 14, 18, 25, 30, 
17, 26, 28, 21, 23, 22, 23, 21, 24, 22, 24, 22, 24, 26, 31, 33, 32, 33, 31, 34, 32, 34, 32, 34\}\\
\{4, 11, 15, 12, 19, 21, 5, 10, 25, 14, 16, 6, 11, 17, 23, 9, 21, 27, 16, 25, 29, 10, 30, 20, 12, 
22, 14, 15, 26, 17, 18, 18, 18, 20, 23, 22, 24, 24, 24, 26, 28, 31, 33, 32, 33, 31, 34, 32, 34, 32, 34\}\\
\{4, 11, 15, 12, 19, 21, 5, 8, 10, 14, 16, 6, 11, 17, 23, 8, 9, 29, 16, 13, 28, 20, 26, 12, 22, 14, 
15, 17, 18, 18, 18, 20, 23, 22, 24, 25, 24, 24, 26, 27, 30, 31, 33, 32, 33, 31, 34, 32, 34, 32, 34\}\\
\{5, 9, 27, 6, 12, 15, 5, 8, 10, 6, 8, 9, 6, 8, 28, 31, 10, 12, 12, 21, 23, 14, 18, 29, 15, 28, 20, 
20, 22, 23, 18, 22, 30, 24, 20, 21, 24, 22, 24, 26, 29, 33, 30, 31, 28, 33, 32, 34, 32, 34, 34\}\\
\{5, 9, 25, 6, 12, 15, 5, 8, 10, 6, 8, 9, 6, 8, 14, 26, 10, 12, 12, 21, 23, 14, 18, 33, 15, 20, 20, 
22, 23, 18, 22, 28, 24, 20, 21, 24, 22, 24, 26, 27, 31, 29, 32, 30, 32, 30, 33, 31, 34, 34, 34\}\\
\{5, 9, 25, 6, 12, 15, 5, 8, 10, 6, 8, 9, 6, 8, 16, 26, 10, 12, 13, 16, 19, 21, 14, 29, 23, 28, 16, 
20, 18, 19, 21, 20, 23, 22, 22, 24, 24, 24, 26, 27, 30, 31, 33, 32, 33, 31, 34, 32, 34, 32, 34\}\\
\{5, 9, 13, 6, 12, 33, 5, 8, 10, 6, 8, 9, 6, 8, 14, 17, 10, 12, 16, 21, 28, 25, 14, 23, 26, 16, 26, 
31, 20, 18, 22, 19, 21, 20, 23, 22, 24, 24, 24, 26, 30, 28, 29, 33, 32, 30, 31, 32, 34, 34, 34\}\\
\end{tiny}
{\ \\}The 25 strong snarks on 36 vertices:\\ 
\begin{tiny}
\{13, 15, 17, 25, 27, 31, 8, 10, 20, 8, 16, 17, 6, 11, 32, 18, 29, 8, 9, 19, 10, 13, 12, 12, 15, 27, 14, 16, 18, 16, 18, 21, 24, 22, 26, 22, 25, 23, 24, 28, 26, 26, 28, 30, 33, 35, 34, 35, 33, 36, 34, 36, 34, 36\}\\
\{13, 15, 17, 12, 25, 27, 8, 10, 20, 8, 16, 17, 6, 11, 31, 18, 30, 8, 9, 19, 10, 13, 12, 12, 15, 14, 16, 18, 16, 18, 21, 24, 22, 26, 22, 25, 23, 24, 28, 26, 26, 28, 32, 29, 33, 35, 34, 35, 33, 36, 34, 36, 34, 36\}\\
\{13, 15, 17, 6, 27, 35, 8, 10, 21, 8, 16, 17, 11, 21, 23, 18, 19, 8, 33, 36, 32, 34, 36, 28, 30, 15, 20, 26, 27, 30, 14, 32, 16, 18, 16, 18, 22, 23, 22, 24, 22, 24, 26, 26, 28, 29, 28, 30, 34, 32, 33, 35, 34, 36\}\\
\{13, 15, 17, 12, 19, 21, 5, 10, 25, 8, 16, 17, 6, 11, 18, 23, 9, 21, 29, 25, 27, 10, 31, 20, 12, 15, 22, 14, 26, 16, 18, 16, 18, 20, 23, 22, 24, 24, 24, 26, 28, 28, 32, 30, 33, 35, 34, 35, 34, 36, 33, 36, 34, 36\}\\
\{13, 15, 17, 12, 21, 29, 5, 8, 10, 8, 16, 17, 6, 11, 18, 27, 8, 9, 19, 10, 13, 22, 12, 31, 20, 14, 16, 18, 16, 28, 18, 23, 25, 24, 25, 23, 26, 24, 26, 24, 26, 28, 35, 33, 30, 32, 33, 36, 34, 36, 34, 35, 34, 36\}\\
\{13, 15, 17, 6, 12, 21, 5, 8, 27, 8, 16, 17, 6, 11, 18, 8, 9, 31, 13, 30, 22, 27, 32, 12, 15, 20, 14, 16, 18, 16, 18, 23, 25, 28, 24, 25, 23, 26, 24, 26, 24, 26, 28, 29, 33, 35, 34, 35, 33, 36, 34, 36, 34, 36\}\\
\{13, 15, 17, 6, 12, 21, 5, 8, 10, 8, 16, 17, 6, 11, 18, 8, 9, 31, 13, 30, 22, 28, 12, 15, 20, 14, 16, 18, 16, 18, 23, 25, 27, 24, 25, 23, 26, 24, 26, 24, 26, 28, 29, 32, 33, 35, 33, 36, 34, 35, 34, 36, 34, 36\}\\
\{5, 9, 13, 15, 27, 31, 8, 10, 35, 6, 8, 36, 6, 35, 19, 8, 16, 17, 10, 36, 25, 13, 16, 18, 20, 29, 32, 14, 15, 17, 16, 18, 30, 21, 24, 22, 24, 22, 23, 25, 26, 27, 26, 26, 33, 30, 32, 33, 30, 34, 32, 34, 34, 36\}\\
\{5, 9, 13, 27, 29, 31, 5, 8, 10, 6, 8, 9, 6, 19, 8, 16, 17, 10, 25, 13, 18, 33, 20, 29, 32, 14, 15, 34, 16, 31, 33, 18, 34, 30, 21, 24, 22, 24, 23, 35, 25, 35, 26, 27, 36, 26, 36, 28, 30, 32, 30, 32, 34, 36\}\\
\{5, 9, 25, 15, 31, 33, 5, 8, 10, 6, 8, 9, 6, 32, 8, 22, 27, 10, 12, 16, 26, 28, 24, 34, 17, 19, 22, 18, 19, 24, 17, 20, 18, 20, 18, 20, 25, 28, 29, 29, 32, 33, 35, 35, 26, 30, 28, 30, 30, 32, 36, 34, 36, 36\}\\
\{5, 9, 13, 12, 15, 25, 5, 8, 10, 6, 8, 9, 6, 29, 8, 14, 19, 10, 12, 20, 26, 31, 30, 14, 18, 25, 29, 32, 17, 26, 28, 21, 23, 22, 23, 21, 24, 22, 24, 22, 24, 26, 30, 33, 35, 34, 35, 30, 33, 36, 34, 36, 34, 36\}\\
\{5, 9, 13, 12, 15, 25, 5, 8, 10, 6, 8, 9, 6, 15, 8, 14, 19, 10, 12, 20, 26, 31, 27, 14, 18, 25, 28, 17, 26, 30, 21, 23, 22, 23, 21, 24, 22, 24, 22, 24, 26, 28, 29, 32, 33, 35, 34, 35, 33, 36, 34, 36, 34, 36\}\\
\{5, 9, 13, 12, 14, 27, 5, 8, 10, 6, 8, 9, 6, 25, 8, 14, 19, 10, 12, 16, 20, 28, 26, 14, 18, 16, 27, 31, 17, 21, 23, 22, 23, 21, 24, 22, 24, 22, 24, 26, 29, 32, 28, 30, 33, 35, 34, 35, 33, 36, 34, 36, 34, 36\}\\
\{5, 9, 13, 12, 14, 15, 5, 8, 10, 6, 8, 9, 6, 25, 8, 14, 19, 10, 12, 16, 27, 29, 26, 14, 18, 16, 28, 17, 21, 23, 22, 23, 21, 24, 22, 24, 27, 22, 24, 26, 30, 31, 28, 32, 33, 35, 34, 35, 33, 36, 34, 36, 34, 36\}\\
\{5, 9, 13, 12, 14, 15, 5, 8, 10, 6, 8, 9, 6, 25, 8, 14, 19, 10, 12, 16, 20, 27, 26, 14, 18, 16, 28, 17, 21, 23, 22, 23, 21, 24, 22, 24, 22, 24, 26, 29, 31, 28, 30, 32, 33, 35, 34, 35, 34, 36, 33, 36, 34, 36\}\\
\{5, 9, 13, 12, 15, 19, 5, 8, 10, 6, 8, 9, 6, 15, 8, 14, 31, 10, 12, 12, 25, 27, 14, 18, 19, 24, 22, 26, 27, 29, 33, 35, 34, 35, 20, 21, 30, 22, 23, 29, 25, 28, 26, 28, 26, 28, 30, 32, 33, 36, 34, 36, 34, 36\}\\
\{5, 9, 13, 12, 15, 19, 5, 8, 10, 6, 8, 9, 6, 15, 8, 14, 17, 10, 12, 12, 25, 27, 14, 18, 19, 24, 22, 26, 27, 18, 31, 30, 20, 21, 32, 22, 23, 29, 25, 28, 26, 28, 26, 28, 33, 35, 34, 35, 33, 36, 34, 36, 34, 36\}\\
\{5, 9, 25, 6, 12, 15, 5, 8, 10, 6, 8, 9, 6, 8, 26, 27, 10, 12, 13, 16, 19, 21, 14, 31, 23, 30, 16, 20, 26, 18, 19, 21, 20, 23, 22, 22, 24, 24, 24, 26, 28, 28, 32, 29, 33, 35, 34, 35, 33, 36, 34, 36, 34, 36\}\\
\{5, 9, 27, 6, 12, 15, 5, 8, 10, 6, 8, 9, 6, 8, 16, 28, 10, 12, 13, 16, 21, 17, 14, 31, 20, 30, 16, 18, 18, 23, 25, 20, 23, 26, 22, 24, 26, 24, 25, 24, 26, 28, 29, 32, 33, 35, 34, 35, 33, 36, 34, 36, 34, 36\}\\
\{5, 9, 27, 6, 12, 15, 5, 8, 10, 6, 8, 9, 6, 8, 14, 28, 10, 12, 13, 16, 21, 17, 14, 35, 18, 18, 22, 20, 30, 18, 19, 23, 25, 24, 25, 23, 26, 24, 26, 24, 26, 28, 29, 33, 31, 34, 32, 34, 32, 35, 33, 36, 36, 36\}\\
\{5, 9, 13, 6, 12, 29, 5, 8, 10, 6, 8, 9, 6, 8, 14, 17, 10, 12, 16, 21, 28, 25, 14, 23, 26, 16, 26, 31, 20, 18, 22, 19, 21, 20, 23, 22, 24, 24, 24, 26, 30, 28, 33, 35, 32, 30, 33, 36, 34, 35, 34, 36, 34, 36\}\\
\{5, 9, 13, 6, 12, 29, 5, 8, 10, 6, 8, 9, 6, 8, 14, 16, 10, 12, 13, 16, 21, 27, 14, 20, 18, 28, 33, 28, 18, 23, 35, 25, 20, 23, 26, 22, 24, 26, 24, 25, 24, 26, 28, 32, 30, 31, 34, 35, 33, 36, 34, 36, 34, 36\}\\
\{5, 9, 25, 6, 12, 15, 5, 8, 10, 6, 8, 9, 6, 8, 14, 26, 10, 12, 12, 21, 23, 14, 18, 27, 15, 20, 20, 22, 23, 18, 22, 28, 24, 20, 21, 24, 22, 24, 26, 29, 31, 28, 32, 30, 33, 35, 34, 35, 34, 36, 33, 36, 34, 36\}\\
\{5, 9, 11, 6, 25, 29, 5, 8, 10, 6, 8, 9, 6, 8, 11, 13, 10, 12, 16, 14, 25, 21, 23, 16, 18, 20, 21, 24, 20, 18, 24, 27, 26, 20, 22, 23, 22, 26, 24, 31, 28, 28, 32, 30, 33, 35, 34, 35, 34, 36, 33, 36, 34, 36\}\\
\{5, 9, 11, 6, 12, 33, 5, 8, 10, 6, 8, 9, 6, 8, 11, 13, 10, 12, 16, 14, 15, 19, 28, 35, 16, 23, 28, 18, 20, 27, 25, 30, 22, 24, 21, 24, 22, 23, 25, 26, 26, 26, 28, 32, 32, 33, 36, 34, 36, 32, 34, 35, 34, 36\}\\
\end{tiny}

%--------------------------------------------------------------------------------------------------------------------------------------------------------------------------------------------
\subsection{Snarks  from Observation \ref{obs:noext}}\label{app:no5cdc}
{\ \\}The 68 snarks on $n=30$ vertices with 2-regular subgraphs which are part of a CDC but not part of any 5-CDC:\\
\begin{tiny}
\{2, 3, 4, 5, 6, 7, 8, 9, 10, 7, 9, 8, 10, 11, 12, 13, 14, 14, 15, 16, 17, 18, 19, 20, 16, 21, 18, 22, 23, 20, 21, 22, 24, 25, 26, 27, 28, 29, 30, 27, 29, 28, 30, 30, 29\}\\
\{2, 3, 4, 5, 6, 7, 8, 9, 10, 7, 9, 8, 10, 11, 12, 13, 14, 14, 15, 16, 17, 18, 19, 20, 16, 21, 18, 22, 23, 20, 21, 24, 25, 22, 26, 27, 28, 27, 29, 29, 30, 28, 30, 30, 29\}\\
\{2, 3, 4, 5, 6, 7, 8, 9, 10, 7, 9, 8, 10, 11, 12, 13, 14, 14, 15, 16, 17, 18, 19, 20, 16, 21, 18, 22, 23, 20, 21, 24, 25, 26, 24, 27, 26, 28, 29, 29, 30, 30, 28, 30, 29\}\\
\{2, 3, 4, 5, 6, 7, 8, 9, 10, 7, 9, 8, 10, 11, 12, 13, 14, 14, 15, 16, 17, 18, 19, 20, 16, 21, 18, 22, 23, 20, 21, 24, 25, 26, 24, 27, 28, 29, 28, 26, 30, 29, 29, 30, 30\}\\
\{2, 3, 4, 5, 6, 7, 8, 9, 10, 7, 9, 8, 10, 11, 12, 13, 14, 14, 15, 16, 17, 18, 19, 20, 16, 21, 18, 22, 23, 20, 21, 24, 25, 26, 26, 27, 28, 29, 25, 28, 30, 29, 28, 30, 30\}\\
\{2, 3, 4, 5, 6, 7, 8, 9, 10, 7, 9, 8, 11, 12, 13, 14, 15, 16, 15, 17, 18, 19, 20, 21, 22, 23, 24, 17, 25, 26, 24, 27, 25, 28, 22, 29, 23, 30, 30, 29, 28, 29, 27, 28, 30\}\\
\{2, 3, 4, 5, 6, 7, 8, 9, 10, 7, 9, 8, 11, 12, 13, 14, 15, 16, 15, 17, 18, 19, 20, 21, 22, 23, 24, 17, 25, 26, 24, 27, 28, 29, 22, 25, 23, 30, 30, 25, 28, 27, 28, 29, 30\}\\
\{2, 3, 4, 5, 6, 7, 8, 9, 10, 7, 9, 8, 11, 12, 13, 14, 15, 16, 15, 17, 18, 19, 20, 21, 22, 23, 24, 17, 25, 26, 24, 27, 28, 29, 22, 28, 23, 30, 30, 28, 29, 27, 29, 27, 30\}\\
\{2, 3, 4, 5, 6, 7, 8, 9, 10, 7, 9, 8, 11, 12, 13, 14, 15, 16, 15, 17, 18, 19, 20, 21, 22, 23, 24, 17, 25, 26, 27, 28, 29, 30, 22, 24, 23, 25, 25, 27, 29, 28, 30, 30, 29\}\\
\{2, 3, 4, 5, 6, 7, 8, 9, 10, 7, 9, 8, 11, 12, 13, 14, 15, 16, 15, 17, 18, 19, 20, 21, 22, 23, 24, 17, 25, 26, 27, 28, 29, 30, 22, 24, 23, 26, 26, 27, 29, 28, 30, 30, 29\}\\
\{2, 3, 4, 5, 6, 7, 8, 9, 10, 7, 9, 8, 11, 12, 13, 14, 15, 16, 17, 18, 15, 19, 20, 21, 22, 23, 24, 19, 25, 24, 26, 27, 28, 27, 22, 29, 23, 30, 30, 29, 28, 26, 28, 29, 30\}\\
\{2, 3, 4, 5, 6, 7, 8, 9, 10, 7, 9, 8, 11, 12, 13, 14, 15, 16, 17, 18, 15, 19, 20, 21, 22, 23, 24, 19, 25, 24, 26, 27, 28, 29, 22, 27, 23, 30, 30, 27, 28, 26, 28, 29, 30\}\\
\{2, 3, 4, 5, 6, 7, 8, 9, 10, 7, 9, 8, 11, 12, 13, 14, 15, 16, 17, 18, 15, 19, 20, 21, 22, 23, 24, 19, 25, 24, 26, 27, 28, 29, 22, 29, 23, 30, 30, 29, 27, 26, 27, 28, 30\}\\
\{2, 3, 4, 5, 6, 7, 8, 9, 10, 7, 9, 8, 11, 12, 13, 14, 15, 16, 17, 18, 15, 19, 20, 21, 22, 23, 24, 19, 25, 26, 27, 28, 29, 30, 22, 24, 23, 30, 30, 26, 28, 27, 29, 29, 28\}\\
\{2, 3, 4, 5, 6, 7, 8, 9, 10, 7, 9, 8, 11, 12, 13, 14, 15, 16, 17, 18, 15, 19, 20, 21, 22, 23, 24, 19, 25, 26, 27, 28, 29, 30, 22, 26, 23, 30, 30, 24, 28, 27, 29, 29, 28\}\\
\{2, 3, 4, 5, 6, 7, 8, 9, 10, 7, 9, 8, 11, 12, 13, 14, 15, 16, 17, 18, 19, 20, 21, 22, 23, 24, 17, 19, 25, 26, 25, 27, 28, 27, 26, 28, 23, 29, 24, 30, 30, 29, 29, 27, 30\}\\
\{2, 3, 4, 5, 6, 7, 8, 9, 10, 7, 9, 8, 11, 12, 13, 14, 15, 16, 17, 18, 19, 20, 21, 22, 23, 24, 17, 19, 25, 26, 25, 27, 28, 27, 26, 29, 23, 28, 24, 30, 30, 28, 29, 27, 30\}\\
\{2, 3, 4, 5, 6, 7, 8, 9, 10, 7, 9, 8, 11, 12, 13, 14, 15, 16, 17, 18, 19, 20, 21, 22, 23, 24, 17, 19, 25, 26, 25, 27, 28, 27, 26, 29, 23, 29, 24, 30, 30, 29, 28, 27, 30\}\\
\{2, 3, 4, 5, 6, 7, 8, 9, 10, 7, 9, 8, 11, 12, 13, 14, 15, 16, 17, 18, 19, 20, 21, 22, 23, 24, 17, 19, 25, 26, 27, 25, 28, 26, 27, 28, 23, 29, 24, 30, 30, 29, 27, 29, 30\}\\
\{2, 3, 4, 5, 6, 7, 8, 9, 10, 7, 9, 8, 11, 12, 13, 14, 15, 16, 17, 18, 19, 20, 21, 22, 23, 24, 17, 19, 25, 26, 27, 25, 28, 26, 27, 29, 23, 28, 24, 30, 30, 28, 27, 29, 30\}\\
\{2, 3, 4, 5, 6, 7, 8, 9, 10, 7, 9, 8, 11, 12, 13, 14, 15, 16, 17, 18, 19, 20, 21, 22, 23, 24, 17, 19, 25, 26, 27, 25, 28, 26, 27, 29, 23, 29, 24, 30, 30, 29, 27, 28, 30\}\\
\{2, 3, 4, 5, 6, 7, 8, 9, 10, 7, 9, 8, 11, 12, 13, 14, 15, 16, 17, 18, 19, 20, 21, 22, 23, 24, 17, 25, 26, 27, 26, 25, 28, 21, 23, 25, 27, 24, 23, 29, 29, 30, 28, 30, 30\}\\
\{2, 3, 4, 5, 6, 7, 8, 9, 10, 7, 9, 8, 11, 12, 13, 14, 15, 16, 17, 18, 19, 20, 21, 22, 23, 24, 17, 25, 26, 27, 26, 25, 28, 25, 27, 26, 29, 23, 29, 24, 30, 30, 29, 28, 30\}\\
\{2, 3, 4, 5, 6, 7, 8, 9, 10, 7, 9, 8, 11, 12, 13, 14, 15, 16, 17, 18, 19, 20, 21, 22, 23, 24, 17, 25, 26, 27, 26, 25, 28, 25, 27, 28, 29, 23, 29, 24, 30, 30, 29, 30, 28\}\\
\{2, 3, 4, 5, 6, 7, 8, 9, 10, 7, 9, 8, 11, 12, 13, 14, 15, 16, 17, 18, 19, 20, 21, 22, 23, 24, 19, 25, 26, 27, 21, 23, 25, 26, 27, 25, 28, 24, 23, 29, 29, 28, 30, 30, 30\}\\
\{2, 3, 4, 5, 6, 7, 8, 9, 10, 7, 9, 8, 11, 12, 13, 14, 15, 16, 17, 18, 19, 20, 21, 22, 23, 24, 19, 25, 26, 27, 25, 26, 27, 28, 27, 25, 29, 23, 28, 24, 30, 30, 28, 29, 30\}\\
\{2, 3, 4, 5, 6, 7, 8, 9, 10, 7, 9, 8, 11, 12, 13, 14, 15, 16, 17, 18, 19, 20, 21, 22, 23, 24, 19, 25, 26, 27, 25, 26, 28, 29, 27, 25, 28, 23, 29, 24, 30, 30, 29, 28, 30\}\\
\{2, 3, 4, 5, 6, 7, 8, 9, 10, 7, 9, 11, 12, 10, 11, 13, 14, 15, 16, 13, 17, 18, 19, 20, 21, 22, 23, 24, 25, 26, 27, 28, 23, 25, 26, 29, 24, 27, 28, 30, 26, 28, 29, 30, 30\}\\
\{2, 3, 4, 5, 6, 7, 8, 9, 10, 7, 9, 11, 12, 10, 11, 13, 14, 15, 16, 17, 18, 19, 20, 21, 22, 23, 24, 17, 23, 25, 23, 26, 21, 27, 22, 28, 28, 29, 27, 30, 27, 29, 28, 30, 30\}\\
\{2, 3, 4, 5, 6, 7, 8, 9, 10, 7, 9, 11, 12, 10, 11, 13, 14, 15, 16, 17, 18, 19, 20, 21, 22, 23, 24, 17, 23, 25, 23, 26, 21, 27, 28, 29, 28, 25, 30, 29, 30, 29, 27, 28, 30\}\\
\{2, 3, 4, 5, 6, 7, 8, 9, 10, 7, 9, 11, 12, 10, 13, 14, 15, 16, 13, 17, 18, 19, 18, 20, 21, 22, 23, 24, 25, 19, 24, 24, 26, 22, 27, 28, 29, 28, 27, 30, 27, 29, 28, 30, 30\}\\
\{2, 3, 4, 5, 6, 7, 8, 9, 10, 7, 9, 11, 12, 10, 13, 14, 15, 16, 13, 17, 18, 19, 18, 20, 21, 22, 23, 24, 25, 19, 26, 26, 27, 22, 24, 23, 28, 28, 29, 30, 27, 29, 30, 28, 30\}\\
\{2, 3, 4, 5, 6, 7, 8, 9, 10, 7, 9, 11, 12, 10, 13, 14, 15, 16, 13, 17, 18, 19, 18, 20, 21, 22, 23, 24, 25, 20, 22, 26, 21, 23, 23, 27, 28, 27, 29, 28, 30, 29, 30, 30, 29\}\\
\{2, 3, 4, 5, 6, 7, 8, 9, 10, 7, 9, 11, 12, 10, 13, 14, 15, 16, 13, 17, 18, 19, 18, 20, 21, 22, 23, 24, 25, 20, 22, 26, 21, 23, 27, 22, 28, 27, 29, 28, 30, 29, 30, 30, 29\}\\
\{2, 3, 4, 5, 6, 7, 8, 9, 10, 7, 9, 11, 12, 10, 13, 14, 15, 16, 13, 17, 18, 19, 18, 20, 21, 22, 23, 24, 25, 20, 26, 27, 21, 22, 22, 23, 28, 26, 29, 28, 30, 30, 29, 30, 29\}\\
\{2, 3, 4, 5, 6, 7, 8, 9, 10, 7, 9, 11, 12, 10, 13, 14, 15, 16, 13, 17, 18, 19, 18, 20, 21, 22, 23, 24, 25, 22, 26, 27, 20, 23, 22, 23, 28, 26, 29, 28, 30, 30, 29, 30, 29\}\\
\{2, 3, 4, 5, 6, 7, 8, 9, 10, 7, 9, 11, 12, 10, 13, 14, 15, 16, 13, 17, 18, 19, 18, 20, 21, 22, 23, 24, 25, 22, 26, 27, 20, 23, 22, 24, 28, 26, 29, 28, 30, 30, 28, 29, 30\}\\
\{2, 3, 4, 5, 6, 7, 8, 9, 10, 7, 9, 11, 12, 10, 13, 14, 15, 16, 13, 17, 18, 19, 18, 20, 21, 22, 23, 24, 25, 24, 26, 27, 25, 28, 22, 28, 23, 29, 29, 28, 30, 26, 27, 30, 30\}\\
\{2, 3, 4, 5, 6, 7, 8, 9, 10, 7, 9, 11, 12, 10, 13, 14, 15, 16, 13, 17, 18, 19, 18, 20, 21, 22, 23, 24, 25, 24, 26, 27, 28, 29, 22, 28, 23, 30, 30, 28, 29, 26, 30, 27, 29\}\\
\{2, 3, 4, 5, 6, 7, 8, 9, 10, 7, 9, 11, 12, 10, 13, 14, 15, 16, 13, 17, 18, 19, 20, 21, 22, 17, 23, 24, 25, 21, 24, 26, 25, 27, 21, 28, 26, 29, 27, 29, 30, 26, 28, 30, 30\}\\
\{2, 3, 4, 5, 6, 7, 8, 9, 10, 7, 9, 11, 12, 10, 13, 14, 15, 16, 13, 17, 18, 19, 20, 21, 22, 17, 23, 24, 25, 21, 24, 26, 27, 28, 21, 29, 26, 27, 28, 30, 27, 29, 30, 30, 29\}\\
\{2, 3, 4, 5, 6, 7, 8, 9, 10, 7, 9, 11, 12, 10, 13, 14, 15, 16, 13, 17, 18, 19, 20, 21, 22, 23, 24, 25, 26, 18, 25, 21, 25, 27, 23, 28, 29, 26, 30, 30, 27, 28, 29, 30, 29\}\\
\{2, 3, 4, 5, 6, 7, 8, 9, 10, 7, 9, 11, 12, 10, 13, 14, 15, 16, 13, 17, 18, 19, 20, 21, 22, 23, 24, 25, 26, 18, 25, 26, 23, 27, 21, 24, 27, 24, 28, 29, 30, 28, 30, 29, 30\}\\
\{2, 3, 4, 5, 6, 7, 8, 9, 10, 7, 9, 11, 12, 10, 13, 14, 15, 16, 13, 17, 18, 19, 20, 21, 22, 23, 24, 25, 26, 18, 25, 26, 27, 28, 21, 23, 27, 23, 29, 28, 30, 30, 29, 30, 29\}\\
\{2, 3, 4, 5, 6, 7, 8, 9, 10, 7, 9, 11, 12, 10, 13, 14, 15, 16, 13, 17, 18, 19, 20, 21, 22, 23, 24, 25, 26, 18, 25, 27, 25, 28, 21, 23, 29, 23, 28, 27, 30, 29, 30, 29, 30\}\\
\{2, 3, 4, 5, 6, 7, 8, 9, 10, 7, 9, 11, 12, 10, 13, 14, 15, 16, 13, 17, 18, 19, 20, 21, 22, 23, 24, 25, 26, 18, 25, 27, 25, 28, 21, 29, 23, 28, 30, 28, 27, 29, 29, 30, 30\}\\
\{2, 3, 4, 5, 6, 7, 8, 9, 10, 7, 9, 11, 12, 10, 13, 14, 15, 16, 13, 17, 18, 19, 20, 21, 22, 23, 24, 25, 26, 18, 25, 27, 25, 28, 21, 29, 23, 28, 30, 28, 27, 30, 29, 30, 29\}\\
\{2, 3, 4, 5, 6, 7, 8, 9, 10, 7, 9, 11, 12, 10, 13, 14, 15, 16, 13, 17, 18, 19, 20, 21, 22, 23, 24, 25, 26, 18, 25, 27, 25, 28, 26, 29, 23, 28, 24, 30, 29, 28, 30, 29, 30\}\\
\{2, 3, 4, 5, 6, 7, 8, 9, 10, 7, 9, 11, 12, 10, 13, 14, 15, 16, 17, 18, 19, 20, 17, 19, 21, 22, 23, 24, 25, 26, 25, 26, 27, 28, 21, 29, 23, 24, 30, 30, 28, 27, 30, 29, 29\}\\
\{2, 3, 4, 5, 6, 7, 8, 9, 10, 7, 9, 11, 12, 10, 13, 14, 15, 16, 17, 18, 19, 20, 17, 21, 19, 22, 23, 24, 25, 26, 23, 20, 25, 27, 26, 24, 28, 29, 30, 28, 29, 30, 29, 28, 30\}\\
\{2, 3, 4, 5, 6, 7, 8, 9, 10, 7, 9, 11, 12, 10, 13, 14, 15, 16, 17, 18, 19, 20, 17, 21, 19, 22, 23, 24, 25, 26, 23, 25, 27, 28, 22, 29, 24, 29, 27, 29, 30, 28, 27, 30, 30\}\\
\{2, 3, 4, 5, 6, 7, 8, 9, 10, 7, 9, 11, 12, 10, 13, 14, 15, 16, 17, 18, 19, 20, 17, 21, 22, 23, 24, 25, 26, 27, 19, 20, 26, 26, 28, 22, 29, 24, 25, 27, 28, 30, 29, 30, 30\}\\
\{2, 3, 4, 5, 6, 7, 8, 9, 10, 7, 9, 11, 12, 10, 13, 14, 15, 16, 17, 18, 19, 20, 17, 21, 22, 23, 24, 25, 26, 27, 19, 20, 26, 26, 28, 24, 29, 24, 27, 28, 30, 28, 29, 30, 30\}\\
\{2, 3, 4, 5, 6, 7, 8, 9, 10, 7, 9, 11, 12, 10, 13, 14, 15, 16, 17, 18, 19, 20, 17, 21, 22, 23, 24, 25, 26, 27, 24, 19, 22, 26, 23, 28, 25, 29, 27, 30, 29, 28, 28, 30, 30\}\\
\{2, 3, 4, 5, 6, 7, 8, 9, 10, 7, 9, 11, 12, 10, 13, 14, 15, 16, 17, 18, 19, 20, 17, 21, 22, 23, 24, 25, 26, 27, 24, 19, 22, 26, 23, 28, 25, 29, 30, 27, 29, 28, 28, 30, 30\}\\
\{2, 3, 4, 5, 6, 7, 8, 9, 10, 7, 9, 11, 12, 10, 13, 14, 15, 16, 17, 18, 19, 20, 17, 21, 22, 23, 24, 25, 26, 27, 24, 19, 26, 22, 26, 28, 25, 29, 30, 27, 28, 29, 28, 30, 30\}\\
\{2, 3, 4, 5, 6, 7, 8, 9, 10, 7, 9, 11, 12, 10, 13, 14, 15, 16, 17, 18, 19, 20, 17, 21, 22, 23, 24, 25, 26, 27, 24, 19, 26, 27, 22, 28, 25, 29, 30, 26, 28, 29, 30, 30, 29\}\\
\{2, 3, 4, 5, 6, 7, 8, 9, 10, 7, 9, 11, 12, 10, 13, 14, 15, 16, 17, 18, 19, 20, 17, 21, 22, 23, 24, 25, 26, 27, 28, 19, 26, 29, 22, 26, 23, 24, 24, 30, 28, 29, 28, 30, 30\}\\
\{2, 3, 4, 5, 6, 7, 8, 9, 10, 7, 9, 8, 11, 10, 12, 13, 14, 14, 15, 16, 17, 18, 19, 20, 16, 21, 18, 22, 23, 20, 21, 22, 24, 25, 26, 27, 28, 29, 30, 27, 29, 28, 30, 30, 29\}\\
\{2, 3, 4, 5, 6, 7, 8, 9, 10, 7, 9, 8, 11, 10, 12, 13, 14, 14, 15, 16, 17, 18, 19, 20, 16, 21, 18, 22, 23, 20, 21, 24, 25, 22, 26, 27, 28, 27, 29, 29, 30, 28, 30, 30, 29\}\\
\{2, 3, 4, 5, 6, 7, 8, 9, 10, 7, 9, 8, 11, 10, 12, 13, 14, 14, 15, 16, 17, 18, 19, 20, 16, 21, 18, 22, 23, 20, 21, 24, 25, 26, 24, 27, 28, 29, 28, 26, 30, 29, 29, 30, 30\}\\
\{2, 3, 4, 5, 6, 7, 8, 9, 10, 7, 9, 8, 11, 10, 12, 13, 14, 14, 15, 16, 17, 18, 19, 20, 16, 21, 18, 22, 23, 20, 21, 24, 25, 26, 26, 27, 28, 29, 25, 28, 30, 29, 28, 30, 30\}\\
\{2, 3, 4, 5, 6, 7, 8, 9, 10, 7, 9, 8, 11, 10, 12, 13, 14, 14, 15, 16, 17, 18, 19, 20, 16, 21, 18, 22, 23, 20, 22, 24, 23, 25, 26, 27, 28, 29, 30, 27, 29, 28, 30, 30, 29\}\\
\{2, 3, 4, 5, 6, 7, 8, 9, 10, 7, 9, 8, 11, 10, 12, 13, 14, 14, 15, 16, 17, 18, 19, 20, 16, 21, 18, 22, 23, 20, 24, 25, 22, 23, 26, 27, 28, 27, 29, 28, 30, 29, 30, 30, 29\}\\
\{2, 3, 4, 5, 6, 7, 8, 9, 10, 7, 9, 8, 11, 10, 12, 13, 14, 14, 15, 16, 17, 18, 19, 20, 16, 21, 18, 22, 23, 20, 24, 25, 22, 26, 27, 26, 24, 28, 29, 28, 30, 30, 28, 29, 30\}\\
\{2, 3, 4, 5, 6, 7, 8, 9, 10, 7, 9, 8, 11, 10, 12, 13, 14, 14, 15, 16, 17, 18, 19, 20, 18, 21, 18, 20, 21, 22, 22, 23, 24, 25, 26, 27, 28, 29, 30, 27, 29, 28, 30, 30, 29\}\\
\{2, 3, 4, 5, 6, 7, 8, 9, 10, 7, 9, 8, 11, 10, 12, 13, 14, 14, 15, 16, 17, 18, 19, 20, 18, 21, 18, 20, 21, 22, 23, 24, 25, 26, 23, 27, 28, 27, 29, 26, 30, 29, 30, 29, 30\}\\
\{2, 3, 4, 5, 6, 7, 8, 9, 10, 7, 9, 8, 11, 10, 12, 13, 14, 14, 15, 16, 17, 18, 19, 20, 18, 21, 18, 20, 22, 23, 22, 24, 24, 25, 26, 27, 28, 29, 30, 27, 28, 29, 30, 29, 30\}\\
\end{tiny}
%--------------------------------------------------------------------------------------------------------------------------------------------------------------------------------------------
\subsection{Snarks  from Observation \ref{obs:noext2}}\label{app:no5cdc2}
{\ \\}The 44 cyclically 5-edge connected snarks on $n=36$ vertices with 2-regular subgraphs which are part of a CDC but not part of any 5-CDC:\\
\begin{tiny}
\{2, 3, 4, 31, 36, 33, 35, 32, 34, 12, 26, 30, 10, 11, 28, 9, 13, 30, 16, 25, 29, 18, 27, 12, 17, 14, 29, 15, 20, 25, 19, 26, 16, 21, 22, 18, 22, 21, 24, 28, 24, 27, 23, 23, 31, 31, 33, 32, 33, 32, 35, 34, 36, 36\}\\
\{2, 3, 4, 31, 36, 33, 35, 32, 34, 12, 26, 30, 10, 11, 28, 9, 13, 30, 18, 25, 29, 16, 27, 12, 17, 14, 29, 15, 20, 25, 19, 26, 16, 21, 22, 18, 22, 21, 24, 28, 24, 27, 23, 23, 31, 31, 33, 32, 33, 32, 35, 34, 36, 36\}\\
\{2, 3, 4, 31, 34, 33, 36, 32, 35, 13, 15, 17, 14, 16, 18, 10, 12, 26, 9, 11, 25, 14, 21, 13, 22, 15, 20, 16, 19, 27, 28, 30, 29, 20, 24, 19, 23, 36, 35, 22, 33, 32, 30, 34, 29, 34, 28, 33, 27, 32, 31, 31, 36, 35\}\\
\{2, 3, 4, 31, 36, 33, 35, 32, 34, 18, 26, 27, 14, 15, 29, 9, 18, 25, 10, 17, 30, 12, 16, 13, 28, 17, 20, 27, 19, 26, 22, 29, 22, 28, 16, 20, 21, 21, 19, 23, 24, 24, 23, 31, 31, 30, 33, 33, 32, 32, 35, 34, 36, 36\}\\
\{2, 3, 4, 31, 36, 33, 35, 32, 34, 10, 17, 28, 13, 25, 27, 8, 15, 26, 16, 29, 11, 15, 30, 12, 29, 14, 18, 19, 27, 18, 22, 21, 26, 21, 17, 22, 20, 20, 24, 28, 23, 24, 23, 31, 31, 30, 33, 33, 32, 32, 35, 34, 36, 36\}\\
\{2, 3, 4, 33, 35, 34, 36, 31, 32, 15, 20, 22, 14, 17, 21, 8, 18, 20, 16, 19, 11, 13, 24, 15, 16, 28, 12, 27, 14, 30, 19, 26, 23, 24, 29, 18, 25, 26, 25, 23, 27, 30, 28, 29, 33, 34, 32, 32, 36, 31, 35, 31, 34, 36\}\\
\{2, 3, 4, 33, 35, 34, 36, 31, 32, 16, 19, 20, 14, 17, 21, 8, 18, 20, 15, 22, 11, 13, 24, 15, 16, 28, 12, 27, 14, 30, 19, 26, 23, 24, 29, 18, 25, 26, 25, 23, 27, 30, 28, 29, 33, 34, 32, 32, 36, 31, 35, 31, 34, 36\}\\
\{2, 3, 4, 31, 36, 33, 35, 32, 34, 9, 10, 30, 11, 14, 29, 8, 9, 14, 11, 13, 12, 16, 25, 18, 15, 25, 17, 26, 17, 21, 30, 19, 28, 22, 20, 27, 20, 23, 24, 22, 24, 23, 31, 31, 33, 29, 32, 28, 33, 32, 35, 34, 36, 36\}\\
\{2, 3, 4, 34, 36, 32, 33, 31, 35, 9, 17, 20, 9, 13, 29, 8, 16, 30, 14, 15, 19, 12, 15, 26, 12, 17, 22, 28, 16, 24, 18, 25, 27, 25, 23, 21, 23, 24, 26, 21, 22, 32, 36, 36, 34, 31, 35, 28, 33, 35, 33, 34, 31, 32\}\\
\{2, 3, 4, 31, 36, 33, 34, 32, 35, 6, 13, 28, 12, 27, 13, 17, 27, 9, 11, 26, 12, 18, 11, 12, 25, 16, 30, 15, 16, 30, 17, 29, 21, 22, 20, 25, 20, 24, 28, 23, 22, 24, 23, 31, 31, 33, 29, 33, 32, 32, 35, 34, 36, 36\}\\
\{2, 3, 4, 31, 36, 33, 34, 32, 35, 9, 25, 28, 7, 18, 26, 8, 10, 11, 15, 12, 30, 16, 30, 14, 29, 19, 27, 16, 21, 27, 20, 26, 17, 21, 22, 22, 25, 20, 29, 24, 28, 24, 23, 23, 31, 31, 33, 33, 32, 32, 35, 34, 36, 36\}\\
\{2, 3, 4, 31, 36, 33, 34, 32, 35, 9, 25, 28, 7, 18, 26, 8, 10, 11, 15, 12, 30, 17, 30, 14, 29, 19, 27, 15, 21, 27, 20, 26, 22, 17, 22, 25, 21, 20, 29, 24, 28, 24, 23, 23, 31, 31, 33, 33, 32, 32, 35, 34, 36, 36\}\\
\{2, 3, 4, 31, 36, 33, 34, 32, 35, 9, 25, 28, 7, 18, 26, 8, 17, 10, 15, 12, 30, 11, 30, 14, 29, 19, 27, 15, 21, 27, 20, 26, 22, 17, 22, 25, 21, 20, 29, 24, 28, 24, 23, 23, 31, 31, 33, 33, 32, 32, 35, 34, 36, 36\}\\
\{2, 3, 4, 31, 36, 32, 34, 33, 35, 19, 21, 23, 20, 22, 23, 10, 14, 29, 9, 13, 24, 10, 12, 11, 16, 27, 18, 25, 17, 25, 15, 27, 22, 26, 19, 29, 20, 28, 21, 24, 20, 22, 36, 32, 34, 28, 34, 33, 33, 31, 31, 32, 35, 36\}\\
\{2, 3, 4, 31, 36, 33, 35, 32, 34, 11, 16, 28, 7, 14, 26, 8, 15, 9, 10, 13, 30, 17, 29, 12, 29, 19, 27, 21, 26, 21, 30, 18, 22, 17, 22, 20, 20, 25, 24, 28, 23, 24, 23, 31, 31, 27, 33, 33, 32, 32, 35, 34, 36, 36\}\\
\{2, 3, 4, 31, 36, 33, 35, 32, 34, 11, 16, 28, 7, 14, 26, 8, 15, 9, 10, 13, 30, 17, 29, 12, 29, 19, 27, 21, 26, 21, 30, 16, 20, 22, 18, 22, 20, 25, 24, 28, 23, 24, 23, 31, 31, 27, 33, 33, 32, 32, 35, 34, 36, 36\}\\
\{2, 3, 4, 31, 36, 33, 34, 32, 35, 6, 13, 28, 10, 27, 13, 18, 27, 11, 12, 26, 10, 12, 15, 11, 17, 16, 30, 16, 25, 30, 17, 29, 22, 21, 20, 25, 22, 24, 28, 21, 23, 24, 23, 31, 31, 33, 29, 33, 32, 32, 35, 34, 36, 36\}\\
\{2, 3, 4, 31, 36, 33, 34, 32, 35, 6, 13, 28, 10, 27, 13, 18, 27, 11, 12, 26, 10, 12, 15, 11, 17, 16, 30, 16, 25, 30, 17, 29, 21, 22, 20, 25, 22, 24, 28, 21, 23, 24, 23, 31, 31, 33, 29, 33, 32, 32, 35, 34, 36, 36\}\\
\{2, 3, 4, 31, 36, 33, 35, 32, 34, 6, 12, 30, 11, 28, 10, 14, 30, 12, 18, 26, 17, 25, 29, 15, 27, 13, 29, 16, 20, 26, 19, 25, 16, 21, 22, 18, 22, 21, 23, 27, 23, 28, 24, 24, 31, 31, 33, 32, 33, 32, 35, 34, 36, 36\}\\
\{2, 3, 4, 31, 36, 33, 34, 32, 35, 6, 9, 26, 8, 28, 9, 13, 27, 10, 30, 11, 16, 26, 18, 29, 17, 25, 30, 17, 29, 15, 16, 25, 21, 28, 19, 22, 20, 27, 20, 23, 24, 22, 24, 23, 31, 31, 33, 32, 33, 32, 35, 34, 36, 36\}\\
\{2, 3, 4, 31, 36, 33, 34, 32, 35, 9, 25, 28, 7, 18, 26, 10, 16, 10, 11, 15, 12, 30, 30, 14, 29, 19, 27, 16, 21, 27, 20, 26, 17, 21, 22, 22, 25, 20, 29, 24, 28, 24, 23, 23, 31, 31, 33, 33, 32, 32, 35, 34, 36, 36\}\\
\{2, 3, 4, 31, 36, 33, 34, 32, 35, 9, 25, 28, 7, 18, 26, 10, 17, 10, 11, 15, 12, 30, 30, 14, 29, 19, 27, 15, 21, 27, 20, 26, 22, 17, 22, 25, 21, 20, 29, 24, 28, 24, 23, 23, 31, 31, 33, 33, 32, 32, 35, 34, 36, 36\}\\
\{2, 3, 4, 31, 36, 33, 35, 32, 34, 6, 9, 30, 8, 13, 8, 11, 14, 12, 16, 25, 16, 28, 29, 18, 27, 18, 26, 15, 25, 17, 26, 21, 30, 19, 22, 28, 20, 20, 23, 24, 22, 24, 23, 31, 31, 33, 32, 29, 32, 33, 35, 34, 36, 36\}\\
\{2, 3, 4, 31, 36, 33, 35, 32, 34, 6, 9, 30, 8, 13, 8, 11, 12, 14, 16, 25, 16, 27, 28, 18, 26, 17, 29, 15, 25, 17, 26, 21, 30, 19, 22, 20, 28, 20, 23, 24, 22, 24, 23, 31, 31, 33, 32, 29, 32, 33, 35, 34, 36, 36\}\\
\{2, 3, 4, 31, 36, 33, 35, 32, 34, 6, 9, 30, 8, 13, 8, 11, 12, 14, 16, 25, 16, 28, 29, 18, 26, 17, 27, 15, 25, 17, 26, 21, 30, 19, 22, 20, 28, 20, 23, 24, 22, 24, 23, 31, 31, 33, 32, 29, 32, 33, 35, 34, 36, 36\}\\
\{2, 3, 4, 31, 36, 32, 35, 33, 34, 6, 7, 30, 21, 29, 27, 28, 20, 21, 26, 18, 19, 21, 11, 12, 26, 20, 25, 19, 24, 17, 18, 22, 16, 18, 23, 16, 17, 19, 22, 27, 24, 36, 29, 33, 32, 28, 34, 31, 34, 33, 32, 31, 35, 36\}\\
\{2, 3, 4, 31, 36, 32, 35, 33, 34, 6, 7, 30, 21, 29, 27, 28, 20, 21, 26, 18, 19, 21, 11, 12, 26, 20, 25, 19, 24, 17, 18, 28, 16, 18, 22, 16, 17, 19, 23, 22, 24, 36, 29, 34, 32, 27, 33, 31, 34, 33, 32, 31, 35, 36\}\\
\{2, 3, 4, 31, 36, 32, 34, 33, 35, 6, 7, 30, 21, 29, 20, 28, 9, 14, 26, 17, 25, 11, 12, 28, 16, 22, 13, 15, 14, 24, 23, 19, 22, 19, 23, 18, 24, 20, 26, 21, 29, 27, 36, 34, 33, 27, 31, 33, 34, 32, 32, 31, 35, 36\}\\
\{2, 3, 4, 31, 36, 32, 34, 33, 35, 6, 7, 30, 21, 29, 20, 28, 9, 13, 26, 17, 25, 11, 12, 28, 16, 22, 14, 15, 16, 23, 23, 24, 19, 22, 19, 18, 24, 20, 26, 21, 29, 27, 36, 34, 33, 27, 31, 33, 34, 32, 32, 31, 35, 36\}\\
\{2, 3, 4, 31, 36, 32, 34, 33, 35, 6, 7, 30, 21, 29, 20, 28, 9, 13, 26, 17, 25, 11, 12, 28, 16, 22, 13, 15, 14, 23, 24, 19, 22, 19, 23, 18, 24, 20, 26, 21, 29, 27, 36, 34, 33, 27, 31, 33, 34, 32, 32, 31, 35, 36\}\\
\{2, 3, 4, 31, 36, 32, 35, 33, 34, 6, 7, 30, 21, 29, 27, 28, 20, 21, 26, 18, 19, 21, 16, 28, 29, 12, 13, 26, 20, 25, 19, 23, 17, 22, 24, 16, 18, 22, 17, 19, 24, 23, 36, 32, 34, 27, 33, 31, 34, 33, 32, 31, 35, 36\}\\
\{2, 3, 4, 33, 36, 31, 34, 32, 35, 6, 8, 16, 9, 11, 10, 11, 15, 12, 27, 12, 21, 17, 29, 24, 22, 14, 20, 25, 26, 28, 23, 30, 19, 22, 23, 24, 19, 20, 21, 31, 34, 33, 36, 36, 35, 29, 31, 30, 32, 28, 32, 35, 34, 33\}\\
\{2, 3, 4, 31, 36, 32, 34, 33, 35, 6, 7, 30, 21, 29, 20, 28, 12, 24, 27, 10, 15, 22, 11, 16, 14, 26, 13, 25, 17, 26, 19, 22, 19, 24, 18, 23, 20, 27, 21, 28, 21, 29, 36, 25, 34, 34, 31, 33, 33, 32, 32, 31, 35, 36\}\\
\{2, 3, 4, 31, 36, 32, 35, 33, 34, 6, 7, 30, 28, 29, 20, 21, 9, 10, 25, 23, 24, 15, 27, 14, 17, 22, 13, 17, 29, 16, 22, 16, 26, 19, 23, 19, 18, 21, 25, 21, 26, 27, 36, 32, 28, 33, 32, 34, 31, 34, 33, 31, 35, 36\}\\
\{2, 3, 4, 31, 36, 33, 35, 32, 34, 6, 11, 30, 12, 28, 10, 14, 30, 11, 16, 26, 17, 25, 29, 15, 27, 18, 13, 29, 20, 26, 19, 25, 16, 21, 22, 18, 22, 21, 23, 27, 23, 28, 24, 24, 31, 31, 33, 32, 33, 32, 35, 34, 36, 36\}\\
\{2, 3, 4, 31, 36, 33, 35, 32, 34, 8, 9, 18, 7, 10, 17, 8, 16, 15, 12, 14, 11, 13, 12, 24, 23, 16, 20, 15, 19, 26, 25, 20, 22, 19, 21, 36, 36, 28, 35, 26, 34, 30, 32, 29, 35, 27, 32, 33, 30, 31, 29, 31, 33, 34\}\\
\{2, 3, 4, 31, 36, 33, 35, 32, 34, 8, 13, 30, 7, 8, 27, 9, 28, 11, 10, 29, 16, 27, 15, 25, 18, 26, 29, 18, 25, 16, 17, 26, 22, 30, 19, 21, 28, 20, 20, 23, 24, 22, 24, 23, 31, 31, 33, 33, 32, 32, 35, 34, 36, 36\}\\
\{2, 3, 4, 31, 36, 33, 34, 32, 35, 7, 13, 28, 9, 18, 26, 12, 30, 9, 14, 30, 15, 11, 16, 18, 15, 29, 19, 27, 16, 20, 17, 20, 21, 22, 22, 25, 21, 24, 28, 23, 24, 23, 31, 31, 27, 33, 29, 33, 32, 32, 35, 34, 36, 36\}\\
\{2, 3, 4, 31, 36, 33, 35, 32, 34, 8, 17, 28, 13, 25, 27, 15, 16, 26, 12, 29, 11, 16, 30, 11, 18, 29, 14, 19, 27, 15, 20, 21, 26, 22, 21, 18, 22, 20, 24, 28, 23, 24, 23, 31, 31, 30, 33, 33, 32, 32, 35, 34, 36, 36\}\\
\{2, 3, 4, 31, 36, 33, 35, 32, 34, 7, 9, 28, 14, 15, 26, 12, 29, 11, 13, 30, 10, 16, 18, 25, 17, 29, 19, 27, 21, 26, 21, 30, 18, 22, 17, 22, 20, 20, 24, 28, 23, 24, 23, 31, 31, 27, 33, 33, 32, 32, 35, 34, 36, 36\}\\
\{2, 3, 4, 31, 36, 33, 34, 32, 35, 7, 14, 28, 13, 27, 29, 12, 30, 9, 15, 30, 10, 16, 18, 26, 17, 18, 25, 19, 27, 15, 20, 17, 22, 22, 21, 25, 20, 21, 24, 28, 23, 24, 23, 31, 31, 33, 29, 33, 32, 32, 35, 34, 36, 36\}\\
\{2, 3, 4, 31, 36, 33, 35, 32, 34, 7, 8, 26, 8, 12, 28, 10, 27, 11, 11, 12, 25, 17, 30, 15, 18, 14, 18, 30, 16, 29, 19, 28, 22, 27, 20, 26, 21, 20, 23, 24, 22, 24, 23, 31, 31, 29, 32, 33, 33, 32, 35, 34, 36, 36\}\\
\{2, 3, 4, 32, 33, 31, 34, 35, 36, 7, 11, 14, 7, 8, 18, 15, 9, 13, 10, 17, 12, 19, 12, 16, 21, 24, 29, 20, 23, 16, 28, 23, 18, 24, 22, 25, 26, 22, 27, 30, 33, 32, 36, 36, 28, 35, 27, 34, 35, 34, 31, 33, 31, 32\}\\
\{2, 3, 4, 32, 33, 34, 36, 31, 35, 7, 11, 14, 7, 8, 18, 15, 9, 13, 10, 17, 12, 19, 12, 16, 21, 24, 29, 20, 23, 16, 28, 23, 18, 24, 22, 25, 26, 22, 27, 30, 33, 32, 36, 36, 28, 35, 27, 34, 35, 34, 31, 33, 31, 32\}\\
\end{tiny}

%--------------------------------------------------------------------------------------------------------------------------------------------------------------------------------------------
\subsection{Counterexamples to Conjectures \ref{conj:zhang}, \ref{conj:ex1}, \ref{conj:ex2} and \ref{FGJ3}}\label{app:ex1}
The 12 cyclically 5-edge connected permutation snarks on 34 vertices:\\
\begin{tiny}
\{9, 11, 17, 6, 10, 25, 8, 12, 31, 8, 9, 16, 21, 26, 27, 16, 27, 11, 23, 33, 23, 10, 12, 14, 18, 19, 24, 29, 22, 34, 16, 17, 29, 20, 22, 32, 20, 26, 30, 22, 31, 24, 34, 26, 28, 28, 33, 30, 32, 32, 34\}\\
\{11, 13, 27, 14, 15, 20, 5, 29, 33, 6, 9, 23, 28, 31, 15, 31, 8, 11, 20, 25, 33, 17, 30, 12, 14, 30, 16, 16, 18, 14, 21, 24, 24, 18, 19, 29, 22, 27, 34, 22, 28, 26, 25, 32, 32, 26, 34, 28, 30, 32, 34\}\\
\{9, 17, 31, 7, 18, 25, 8, 12, 26, 6, 8, 9, 13, 19, 26, 15, 19, 8, 11, 23, 21, 23, 29, 16, 17, 16, 29, 14, 27, 21, 25, 18, 33, 22, 24, 20, 34, 24, 34, 22, 30, 32, 32, 28, 28, 28, 31, 30, 33, 32, 34\}\\
\{21, 23, 33, 14, 15, 19, 5, 27, 29, 6, 9, 25, 31, 34, 15, 31, 11, 19, 28, 18, 25, 28, 10, 30, 12, 14, 16, 23, 22, 29, 14, 17, 24, 26, 22, 26, 18, 33, 20, 20, 27, 22, 30, 24, 34, 32, 32, 28, 30, 32, 34\}\\
\{9, 15, 19, 7, 10, 23, 12, 16, 21, 8, 17, 22, 6, 19, 31, 13, 23, 8, 11, 16, 17, 27, 12, 33, 14, 25, 24, 22, 32, 18, 29, 16, 25, 20, 20, 26, 28, 28, 22, 31, 34, 26, 34, 30, 30, 28, 33, 30, 32, 32, 34\}\\
\{15, 17, 31, 13, 19, 33, 5, 8, 30, 8, 9, 25, 15, 32, 13, 25, 34, 8, 22, 29, 10, 27, 23, 26, 17, 22, 24, 14, 21, 26, 16, 16, 24, 16, 18, 20, 32, 20, 23, 34, 22, 30, 24, 28, 28, 28, 31, 30, 33, 32, 34\}\\
\{9, 17, 21, 7, 18, 23, 5, 12, 16, 8, 13, 29, 6, 31, 23, 25, 11, 33, 20, 26, 10, 27, 12, 18, 14, 15, 24, 14, 25, 24, 21, 34, 20, 32, 19, 28, 28, 20, 31, 22, 26, 33, 24, 30, 30, 28, 29, 30, 32, 34, 34\}\\
\{9, 15, 31, 13, 32, 33, 21, 24, 34, 6, 9, 27, 6, 15, 24, 19, 8, 11, 23, 18, 34, 10, 26, 32, 14, 17, 14, 26, 29, 28, 30, 16, 22, 22, 30, 20, 25, 20, 27, 20, 28, 22, 29, 24, 33, 26, 31, 28, 30, 32, 34\}\\
\{11, 19, 25, 10, 13, 17, 8, 12, 31, 6, 8, 9, 15, 18, 33, 23, 29, 11, 24, 26, 30, 10, 25, 16, 27, 14, 16, 22, 23, 22, 32, 16, 19, 18, 26, 20, 34, 21, 34, 22, 31, 24, 30, 28, 28, 28, 32, 30, 33, 32, 34\}\\
\{9, 15, 33, 6, 7, 22, 12, 27, 29, 8, 17, 23, 19, 25, 29, 13, 25, 8, 28, 16, 10, 17, 22, 34, 14, 15, 28, 31, 34, 14, 32, 18, 16, 27, 26, 24, 31, 20, 21, 24, 30, 22, 33, 24, 32, 26, 30, 28, 30, 32, 34\}\\
\{9, 11, 19, 7, 13, 25, 17, 20, 29, 6, 9, 23, 15, 17, 31, 13, 31, 8, 21, 23, 29, 10, 16, 25, 21, 33, 14, 27, 30, 18, 32, 33, 16, 26, 27, 18, 32, 20, 34, 24, 22, 26, 34, 24, 30, 28, 28, 28, 30, 32, 34\}\\
\{9, 13, 17, 6, 10, 25, 5, 12, 31, 6, 8, 9, 6, 19, 14, 32, 33, 23, 33, 34, 12, 29, 12, 22, 32, 19, 27, 27, 30, 18, 21, 26, 24, 26, 29, 22, 28, 20, 23, 20, 28, 22, 31, 24, 34, 26, 30, 28, 30, 32, 34\}\\
 \end{tiny}

\end{document}